\documentclass[draft]{amsart}
\usepackage{dsfont}
\usepackage{amsfonts}
\usepackage{mathrsfs}
\usepackage{amssymb}
\usepackage{amsmath}
\usepackage{amsfonts}
\usepackage{amsfonts,enumerate}
\usepackage{pifont}
\usepackage{enumerate}
\usepackage{graphics}
\usepackage{verbatim}
\usepackage{amsthm}
\usepackage{latexsym,bm}
\usepackage{color}
\numberwithin{equation}{section}
\newtheorem{theorem}{Theorem}[section]
\newtheorem{corollary}[theorem]{Corollary}
\newtheorem{lemma}[theorem]{Lemma}
\newtheorem{proposition}[theorem]{Proposition}

\theoremstyle{definition}
\newtheorem{remark}[theorem]{Remark}
\newtheorem{definition}[theorem]{Definition}

\newcommand{\M}{{\mathcal M}}

\begin{document}

\title[]{Noncommutative maximal ergodic theorems for spherical means on the Heisenberg group}

\author{Guixiang Hong}
\address{School of Mathematics and Statistics, Wuhan University, Wuhan 430072, China\\
\emph{E-mail address: guixiang.hong@whu.edu.cn}}

\thanks{\small {{\it MR(2010) Subject Classification}.} Primary 46L55, 37A15; Secondary 46L51, 46B75.}
\thanks{\small {\it Keywords.}
Von Neumann algebra, noncommutative $L_p$ space, tensor/crossed product, spherical means, Heisenberg group, maximal ergodic theorem, individual ergodic theorem.}

\maketitle

\begin{abstract}
We prove maximal ergodic theorems for spherical averages on the Heisenberg  groups acting on $L_p$ spaces over measure spaces not necessarily commutative, that is, on noncommutative $L_p$ spaces. The scale of $p$ is optimal in the reduced Heisenberg group case. We also obtain the corresponding individual ergodic theorems and differential theorems in the noncommutative setting. The results can be regarded as noncommutative analogues of Nevo-Thangavelu's ergodic theorems. The approach of proof involves recent developments in noncommutative $L_p$ spaces and in noncommutative harmonic analysis, in addition to the spectral theory used in the commutative setting.
\end{abstract}

\section{Introduction}

Let $H^n$ denote the Heisenberg group, and let $\sigma_r$ denote the normalized Lebesgue measure on the sphere $\{(z,0): |z|=r\}$. Let $(X,m)$ be a standard measure space on which $H^n$ acts measurably by measure preserving transformations, and let $\pi(\sigma_r)$ denote the operator canonically associated with $\sigma_r$ on $L_p(X)$.  Nevo and Thangavelu \cite{NeTh97}  established the maximal ergodic theorems in $L_p$ for the sequence of operators $(\pi(\sigma_r))_{r>0}$.  Their result was later improved by Narayanan and Thangavelu \cite{NaTh04}  to obtain the optimal scale of $p$.

Along the current research line in noncommutative harmonic analysis, it is interesting to understand whether the maximal ergodic theorems remain true in noncommutative $L_p$ space for the two sequences of operators $(\pi(\sigma_r)\otimes id_{B(\ell_2)})_{r>0}$ (the tensor product extension) and $(\pi(\sigma_r)\rtimes id_{G})_{r>0}$ (the crossed product extension), where $B(\ell_2)$ is the matrix algebra of infinite dimension and $G$ is some locally compact group (or discrete group). We refer the readers to the next section for all the notions not appearing here and below.

The two kinds of extension problems are not trivial at all in many cases. The tensor product problems are motivated by matrix-valued harmonic analysis and quantum probability. To deal with the operator-valued Hardy-Littlewood maximal inequality, Mei \cite{Mei07} invented a very interesting and powerful techniques to reduce it to several noncommutative Doob's maximal inequalities. To prove  the weak type $(1,1)$ estimates of Calder\'on-Zygmund operators acting on matrix-valued functions,  Parcet \cite{Par09} exploited the pseudo-localization principle of the operators themselves and built the noncommutative Calder\'on-Zygmund decomposition. While the crossed product problems are firstly introduced by Junge, Mei and Parcet \cite{JMP1} in order to study Fourier multipliers on group von Neumann algebra. To  obtain the mapping property of twisted Hilbert transform, Parcet and Rogers \cite{PaRo} combined the arguments from Kakeya type constructions with geometric group theory and noncommutative Littlewood-Paley methods. To establish the Poincar\'e type inequalities for  group measure spaces,  Zeng \cite{Zen14} made use of a twisted version of Khintchine inequality with optimal order of constant.

In the present paper, we establish the maximal ergodic theorem associated to spherical averages on Heisenberg groups in a more general framework,  including the previously mentioned two extensions. In this sense, this paper is the first attempt to build  maximal ergodic theorem in noncommutative $L_p$ space for general trace preserving actions of locally compact groups, being a part of effort \cite{JuXu06} \cite{Ana06} \cite{Hu08} to extend classical maximal ergodic theorems to the noncommutative setting. In \cite{JuXu06}, Junge and Xu established the noncommutative analogue of the classical Dunford-Schwartz and Stein's maximal ergodic inequality. Anantharaman \cite{Ana06} and Hu \cite{Hu08} respectively proved the noncommutative analogues of Nevo-Stein's maximal ergodic inequality for free groups using two different approaches. Junge and Xu's maximal ergodic theorems have also been extended to more general semigroups---for instance analytic semigroups \cite{Bek08} \cite{LeXu12}, and to more general function spaces---for instance Orlicz spaces \cite{BCO17} \cite{Dir15}.

Motivated by the developments in classical ergodic theory, along the research line of noncommutative maximal ergodic theorem for trace preserving actions of locally compact groups, in a subsequent paper \cite{Hon1}, we study ergodic theorems for spherical averages on Gelfand pairs. As an application, the spherical maximal ergodic inequality could be used to obtain dimension free estimates of maximal inequality associated to ball averages acting on noncommutative $L_p$ spaces. For instance, in another subsequent paper \cite{Hon}, we obtain dimension free estimate associated to ball averages in $\mathbb R^n$ action on noncommutative $L_p$ spaces.

In the next section, we firstly recall the definition of vector-valued noncommutative $L_p$ spaces $L_p(\ell_{\infty})$-a substitute of the $L_p$-norm of maximal function in the commutative case, and the crossed product construction. Then we describe the framework for the research of noncommutative maximal ergodic theorem for trace-preserving action of locally compact groups. Finally, we state the results obtained in the present paper.

\section{Preliminaries, framework, and state of results}

\subsection{Vector-valued noncommutative $L_p$-spaces}
Let $\mathcal{M}$ be a von Neumann algebra equipped with a normal
semifinite faithful trace $\tau$ and $S^+_{\mathcal{M}}$ be the set
of all positive element $x$ in $\mathcal{M}$ with
$\tau(s(x))<\infty$, where $s(x)$ is the smallest projection $e$
such that $exe=x$. Let $S_{\mathcal{M}}$ be the linear span of
$S^+_{\mathcal{M}}$. Then any $x\in S_{\mathcal{M}}$ has finite
trace, and $S_{\mathcal{M}}$ is a $w^*$-dense $*$-subalgebra of
$\mathcal{M}$.

Let $1\leq p<\infty$. For any $x\in S_{\mathcal{M}}$, the operator
$|x|^p$ belongs to $S^+_{\mathcal{M}}$ ($|x|=(x^*x)^{\frac{1}{2}}$).
We define
$$\|x\|_p=\big(\tau(|x|^p)\big)^{\frac{1}{p}},\qquad\forall x\in S_{\mathcal{M}}.$$
One can check that $\|\cdot\|_p$ is well defined and is a norm on
$S_{\mathcal{M}}$. The completion of $(S_{\mathcal{M}},\|\cdot\|_p)$
is denoted by $L_p(\M)$ which is the usual noncommutative $L_p$-
space associated with $(\M,\tau)$. For convenience, we usually set
$L_{\infty}(\M)=\M$ equipped with the operator norm
$\|\cdot\|_{\M}$. We refer the reader to \cite{PiXu03} for more
information on noncommutative $L_p$-spaces.

\medskip

Let us recall the definition of the noncommutative maximal norm
introduced by Pisier \cite{Pis98} and Junge \cite{Jun02}. We define
$L_p(\M;\ell_\infty)$ to be the space of all sequences
$x=(x_n)_{n\ge1}$ in $L_p(\M)$ which admit a factorization of the
following form: there exist $a, b\in L_{2p}(\M)$ and a bounded
sequence $y=(y_n)$ in $L_\infty(\M)$ such that
 $$x_n=ay_nb, \quad \forall\; n\geq1.$$
The norm of  $x$ in $L_p(\M;\ell_\infty)$ is given by
 $$\|x\|_{L_p(\M;\ell_\infty)}=\inf\big\{\|a\|_{2p}\,
 \sup_{n\geq1}\|y_n\|_\infty\,\|b\|_{2p}\big\} ,$$
where the infimum runs over all factorizations of $x$ as above.

We will follow the convention adopted in  \cite{JuXu06}  that
$\|x\|_{L_p(\M;\ell_\infty)}$ is denoted by
 $\big\|\sup_n^+x_n\big\|_p\ .$ We should warn the reader that
$\big\|\sup^+_nx_n\big\|_p$ is just a notation since $\sup_nx_n$
does not make any sense in the noncommutative setting. We find,
however, that $\big\|\sup^+_nx_n\big\|_p$ is more intuitive than
$\|x\|_{L_p(\M;\ell_\infty)}$. The introduction of this notation is
partly justified by the following remark.

\begin{remark}
Let $x=(x_n)$ be a sequence of self-adjoint operators in $L_p(\M)$.
Then $x\in L_p(\M;\ell_\infty)$ if and only if there exists a positive element
$a\in L_p(\M)$ such that $-a\leq x_n \le a$ for all $n\ge1$. In this
case we have
 $$\big \|{\sup_{n \ge1}}^{+} x_n \big \|_p=\inf\big\{\|a\|_p\;:\; a\in L_p(\M),\; -a\leq x_n\le a,\;\forall\; n\geq1\big\}.$$
\end{remark}

More generally, if $\Lambda$ is any index set, we define  $L_p (\M;
\ell_{\infty}(\Lambda))$ as the space of all $x =
(x_{\lambda})_{\lambda \in \Lambda}$ in $L_p (\M)$ that can be
factorized as
 $$x_{\lambda}=ay_{\lambda} b\quad\mbox{with}\quad a, b\in L_{2p}(\M),\; y_{\lambda}\in L_\infty(\M),\; \sup_{\lambda}\|y_{\lambda}\|_\infty<\infty.$$
The norm of $L_p (\M; \ell_{\infty}(\Lambda))$ is defined by
 $$\big \| {\sup_{{\lambda}\in\Lambda}}^{+} x_{\lambda}\big \|_p=\inf_{x_{\lambda}=ay_{\lambda} b}\big\{\|a\|_{2p}\,
 \sup_{{\lambda}\in\Lambda}\|y_{\lambda}\|_\infty\,\|b\|_{2p}\big\} .$$
It is shown in \cite{JuXu06} that $x\in L_p (\M;
\ell_{\infty}(\Lambda))$ if and only if
 $$\sup\big\{\big \| {\sup_{{\lambda}\in J}}^{+} x_{\lambda}\big \|_p\;:\; J\subset\Lambda,\; J\textrm{ finite}\big\}<\infty.$$
In this case, $\big \| {\sup_{{\lambda}\in\Lambda}}^{+}
x_{\lambda}\big \|_p$ is equal to the above supremum.

A closely related operator space is $L_p(\mathcal{M};\ell^c_{\infty})$ for $p\geq2$ which is the set of all sequences $(x_n)_n\subset L_p(\mathcal{M})$ such that
$$\|{\sup_{n\geq1}}^+|x_n|^2\|^{1/2}_{p/2}<\infty.$$
While $L_p(\mathcal{M};\ell^r_{\infty})$ for $p\geq2$ is the Banach space of all sequences $(x_n)_n\subset L_p(\mathcal{M})$ such that $(x^*_n)_n\in L_p(\mathcal{M};\ell^c_{\infty})$. All these spaces fall into the scope of amalgamated $L_p$ spaces intensively studied in \cite{JuPa10}. What we need about these spaces is the following interpolation results.
\begin{lemma}\label{lem:l8c+l8r}
Let $2\leq p\leq\infty$. Then we have
$$(L_{p}(\mathcal{M};\ell^c_{\infty}),L_{p}(\mathcal{M};\ell^r_{\infty}))_{1/2}=L_{p}(\mathcal{M};\ell_{\infty})$$
with equivalent norms.
\end{lemma}
We refer the reader to \cite{Jun02}, \cite{Mus03} and \cite{JuPa10} for more properties on these spaces.

\subsection{$W$*-dynamical system}
Let $G$ be a locally compact second countable (lcsc) group. Let $(\M,\tau)$ be a von Neumann algebra equipped with a semifinite normal faithful trace $\tau$.  A $W$*-dynamical system is a quadruple $(\M,\tau,G,\gamma)$ in which $\gamma$ is a trace-preserving continuous homomorphism of $G$ into the group $Aut(\M)$ of automorphisms of $\M$ in the weak* topology. This means that for each $x\in\M$, the $\M$-valued function defined on $G$: $g\rightarrow\gamma(g)x$ is  $\sigma$-weakly continuous, and $\tau(\gamma(g)x)=\tau(x)$ for any $g\in G$ and $x\in\M\cap L_1(\M)$. Such homomorphism $\gamma$ will be called an action.

\medskip

Start with a $W$*-dynamical system $(\M,\tau,G,\gamma)$, we can build a new von Neumann algebra. Consider the representation $\rho$ of $\M$ on $L_2(G;L_2(\mathcal M))$ given by
$$\rho(x)\xi(g)=\gamma(g^{-1})(x)\xi(g).$$
Then the crossed product of $\M$ by $G$, denoted by $\M\rtimes_{\gamma}G$, is defined as the weak operator closure of $id_{\M}\otimes \lambda(G)$ and $\rho(\M)$ in $B(L_2(G;L_2(\mathcal M)))$, where $\lambda$ is the left regular representation of $G$. Trivially $\M$ and $\mathcal{L}(G)$, the group von Neumann algebra, can be viewed as von Neumann subalgebras of $\M\rtimes_{\gamma}G$. In this paper, we assume $G$ is unimodular, or equivalently, that the Haar measure is inversion invariant.  A generic element of $\M\rtimes_{\gamma}G$ can be formally written as
\begin{align*}
\int_{G}x_g\rtimes_{\gamma}\lambda(g)dg&=\int_G\rho(x_g)(id_{\M}\otimes\lambda(g))dg\\
&=\int_{G\times G\times G}(\gamma(h)x_g\otimes e_{h,h})(1_\M)\otimes e_{gh',h'}dgdhdh'\\
&=\int_{G\times G}\gamma(h^{-1})x_g\otimes e_{h,g^{-1}}dgdh\\
&=\int_{G\times G}\gamma(g^{-1})x_{gh^{-1}}\otimes e_{g,h}dgdh.
\end{align*}
There is a canonical trace on $\M\rtimes_{\gamma}G$ given by
$$\tau\rtimes\tau_G(x\rtimes\lambda(g))=\tau\otimes\tau_G(x\otimes\lambda(g))=\tau(x)\delta_{g=e},$$
where we denote by $\tau_G$ the canonical trance on $\mathcal{L}(G)$. The arithmetic in $\M\rtimes_\gamma G$ is given by
$$(x\rtimes_\gamma \lambda(g))^*=\gamma(g^{-1})(x^*)\rtimes_\gamma \lambda(g^{-1})$$
and
$$(x\rtimes_\gamma \lambda(g))(y\rtimes_\gamma \lambda(h))=(x\gamma(g)(y))\rtimes_\gamma \lambda(gh).$$

The group measure space mentioned in the introduction refers to a special case of the crossed product, that is, $\M=L_{\infty}(X,\mu)$. The crossed product can be viewed as a twisted version of the tensor product $\M\overline{\otimes}\mathcal{L}(G)$, but is very different in the sense that under the conditions that $\M$ is a commutative von Neumann algebra and $G$ is an abelian group, the tensor product remains a commutative algebra but the crossed product is not a commutative algebra any more.
\medskip

\subsection{Framework}

We formulate the framework for the study of the maximal ergodic theorem in noncommutative $L_p$ space for trace-preserving action of locally compact groups in the present (and subsequent) paper. Let $(\M,\tau,G,\alpha)$ be a $W$*-dynamical system with $\M$ equipped with a normal faithful trace $\tau$, and the action $\alpha$ being trace preserving. Then $\alpha$ is naturally extended to isometric automorphisms of $L_p(\M)$ for all $1\leq p\leq\infty$, which is still denoted by $\alpha$. As it is well-known, the weak*-continuity of $(\alpha(g))_{g\in G}$ on $\M$ induces the strong continuity of  $(\alpha(g))_{g\in G}$, i.e. for each $x\in L_p(\M)$ with $1\leq p<\infty$, the map $g\rightarrow \alpha(g)x$ is a continuous map from $G$ to $L_p(\M)$, where we take the norm topology on $L_p(\M)$.

Let $M(G)$ denote the Banach space of bounded complex Borel measures on $G$. For each $\mu\in M(G)$, there corresponds an operator $\alpha(\mu)$, with norm bounded by $\|\mu\|_1$ in every $L_p(\M)$, $1\leq p\leq\infty$, given by
$$\alpha(\mu) x=\int_G\alpha(g)xd\mu(g),\;\forall x\in L_p(\M).$$
This definition should be justified as follows. For any $ x\in L_p(\M)$ and $ y\in L_q(\M)$ where $1/p+1/q=1$, the function $g\rightarrow\langle \alpha(g)x,y\rangle$ is continuous on $G$, bounded by $\|x\|_p\|y\|_q$. Hence
$$\int_G\langle\alpha(g)  x,y\rangle d\mu(g)\leq\|\mu\|_1\|x\|_p\|y\|_q.$$
It follows the operator $\alpha(\mu) x$ is well defined and $\alpha(\mu) x$ is in $L_p(\M)$. Moreover, $\|\alpha(\mu) x\|_p\leq\|\mu\|_1\|x\|_p$ so that $\|\alpha(\mu)\|\leq\|\mu\|_1$.

As $G$ having been assumed to be unimodular. It is easy to check that $\mu\rightarrow \alpha({\mu})$ is a norm-continuous $*$-representation of the involutive Banach algebra $M(G)$ as an algebra of operators on $L_2(\M)$. We recall that the product in $M(G)$ is defined as convolution $\mu\ast\nu(f)=\int_G\int_Gf(gh)d\mu(g)d\nu(h)$, and the involution is $\mu^*(E)=\overline{\mu(E)^{-1}}$.

Denote by $P(G)$ the subset of probability measures in $M(G)$. Let $t\rightarrow v_t$ be a weakly continuous map from $\mathbb{R}_+$ to $P(G)$, namely $t\rightarrow v_t(f)$ is continuous for each $f\in C_c(G)$. We will refer to $(v_t)_{t>0}$ as a one-parameter family of probability measures. We can now formulate the following.

\begin{definition}\label{def:nc pointwise ergodic family}
A one-parameter family $(v_t)_{t>0}\subset P(G)$ will be called a global (resp. local) noncommutative maximal ergodic family in $L_p$ if for every trace preserving action $\alpha$ of $G$  on any noncommutative measure space $(\M,\tau)$, there exists a constant $c_p$ such that
$$\|{\sup_{t>1}}^+\alpha(v_t)x\|_p\leq c_p\|x\|_p\;(\mathrm{resp.}\;\|{\sup_{0<t\leq1}}^+\alpha(v_t)x\|_p\leq c_p\|x\|_p\;).$$
\end{definition}

\begin{remark}\label{re:birkhoff}
The basic example is the one-parameter family $(\mu_t)_{t>0}$ of probability measures on $G=\mathbb{R}$ given by
$\mu_t=\frac 1t\int^t_0\frac{1}{2}(\delta_s+\delta_{-s})ds$. This family is a both global and local noncommutative maximal ergodic family in every $L_p$ with $1<p<\infty$ which can be deduced from Junge and Xu's result \cite{JuXu06} on the noncommutative analogue of Dunford-Schwartz maximal ergodic theorem for semigroups (see also Lemma \ref{lem:p estimate for Mt} below) since the operator $\alpha(\mu_t)$ can be viewed as an average of two ergodic averages  $M_t$'s associated to the semigroups $(T_+^s)_{s>0}$ and  $(T_{-}^s)_{s>0}$, where $T_+^sx=\alpha_sx$ and $T_-^sx=\alpha_{-s}x$.

Also note that $\mu_t$ is the uniform average of the natural spherical measures $\sigma_s=\frac{1}{2}(\delta_s+\delta_{-s})$ on a sphere of radius $s$ with center 0.
\end{remark}

\subsection{Statement of results}
Let  $H=H^n=\mathbb{C}^n\times\mathbb{R}$ denote the Heisenberg group with group operation
$$(z,t)(w,s)=(z+w,t+s+\frac{1}{2}Im (z\cdot\bar{w})).$$
Let $S^n_{r,t}=\{(z,t):\;|z|=r\}$ be the sphere in the Heisenberg group $H^n$.  The first one-parameter family of probability measures we are concerned in the present paper is the normalized surface measures $\sigma_r$'s on the sphere $S^n_r:=S^n_{r,0}$.
The first main result of this paper is stated as follows:

\begin{theorem}\label{thm:maximal radial average}
The spherical means $\sigma_r$ on $H^n$ is a both global and local noncommutative maximal ergodic families in $L_p$, provided $n>1$ and $(2n-1)/(2n-2)<p\leq\infty$.
\end{theorem}


As it is well-known that the scale of $p$ in Theorem \ref{thm:maximal radial average} may not be optimal, since the best lower bound of $p$ is $2n/(2n-1)$ whenever the underlying von Neumann algebra is commutative. Narayanan and Thangavelu obtained the optimal result first on $\mathbb{H}^n$, then on general measure spaces using transference principle. Even though the transference principle has been establish by the author in \cite{Hon}, another main tool used by Narayanan and Thangavelu---Fefferman-Stein vector-valued maximal inequality---remains an open problem in the noncommutative setting. So we can not improve the lower bound of $p$ to $2n/(2n-1)$. However, in the reduced Heisenberg group case, we do obatin the optimal lower bound. 

The center of $H^n$ is denoted by $Z(H^n)=\{(0,t):t\in\mathbb{R}\}$ which is isomorphic to $\mathbb{R}$. The reduced Heisenberg group is defined as $\bar{H}^n=H^n/Z_1$, where $Z_1=\{(0,n):n\in\mathbb{Z}\}$ is the discrete subgroup of $Z(H^n)$. Another one-parameter family of  probability measures we are concerned is the spherical means $\bar{\sigma}_r$ on $\bar{S}^n_{r}:=\bar{S}^n_{r,0}=\{(z,0):\;|z|=r\}$.

\begin{theorem}\label{thm:maximal radial average reduced}
The spherical means $\bar{\sigma}_r$ on $\bar{H}^n$ is a global noncommutative maximal ergodic families in $L_p$, provided $n>1$ and $2n/(2n-1)<p\leq\infty$.
\end{theorem}

Both the proof of the two maximal ergodic theorems depend on the spectral theory. That the related spectral estimates is not sharp stop us from obtaining optimal scale in Theorem {thm:maximal radial average}; to obtain the sharp estimates in the reduced Heisenberg roup, we use the discreteness of the Largurre spectrum to sharpen the estimates and optimal scale of $p$ of spherical maximal inequality in $\mathbb R^{2n}$ \cite{Hon14}.  

Based on the two results, we get the maximal ergodic theorem for the tensor/crossed product extensions.
Let $(X,m)$ be a standard measure space on which $H^n$ acts measurably by measure preserving transformations, and let $\pi(\sigma_r)$ denote the operator canonically associated with $\sigma_r$ on $L_p(X)$.

\begin{corollary}\label{cor:product extension theorem}
Let $n>1$. We have the following results:
\begin{enumerate}[\rm (i)]
\item The triple $(L_{\infty}(X)\overline{\otimes}\M,H^n,\pi\otimes id_\M)$ is a $W$*-dynamical system. Hence we have the maximal ergodic theorem in $L_p$ with $(2n-1)/(2n-2)<p\leq\infty$, that is, there exists a universal constant $C_p>0$ such that
    \begin{align*}
      \|{\sup_{r>0}}^+\pi({\sigma}_r)\otimes id_\M(f)\|_p\leq C_{p}\|f\|_p,\;\forall f\in L_p(L_{\infty}(X)\overline{\otimes}\M).
    \end{align*}
\item The triple $(L_{\infty}(X)\rtimes_\gamma G,H^n,\pi\rtimes_\gamma id_G)$ is a $W$*-dynamical system provided that $\pi$ and $\gamma$ commute, that is, $\pi(\ell)\gamma(g)=\gamma(g)\pi(\ell)$ for any $\ell\in H^n$ and $g\in G$. Hence we have the maximal ergodic theorem in $L_p$ with $(2n-1)/(2n-2)<p\leq\infty$.
\item The triple $(L_{\infty}(H^n)\rtimes_R U(n),H^n,T\rtimes_R id_{U(n)})$ is not a $W$*-dynamical system, where $T$ denotes the translation action and $R$ denotes the rotation action. But we still have the maximal ergodic theorem in $L_p$ with $(2n-1)/(2n-2)<p\leq\infty$.
\item Simiar statements in \rm{(i)-(iii)} hold for $2n/(2n-1)<p\leq\infty$ by replacing $H^n$ with $\bar{H}^n$.
\end{enumerate}
\end{corollary}

Corollary \ref{cor:product extension theorem} (i) and (ii) follow trivially from Theorem \ref{thm:maximal radial average}. And (iii) follows from the easy fact that the operators $T(\sigma_r)\rtimes_R id_{U(n)}$ is the restriction of $T(\sigma_r)\otimes id_{B(L_2(U(n)))}$ to $L_{\infty}(H^n)\rtimes_R U(n)$ and (i).

The rest part of the paper is organized as follows. In next section, we recall the spherical spectral theory on Heisenberg group which plays essential role in our proof. In Section 4, we recall the mean ergodic theorem, which will be used for maximal and individual ergodic theorems. Section 5 (resp. Section 6) is devoted to the proof of Theorem \ref{thm:maximal radial average} (resp. Theorem \ref{thm:maximal radial average reduced}).
As in \cite{JuXu06} and \cite{Hu08}, the noncommutative maximal ergodic theorems may be used to decude the corresponding noncommutative individual ergodic theorems. In the last section, we establish the individual ergodic theorems. Throughout this paper, $C$ denotes a universal constant, may varying
from line to line.

\section{Spherical spectral theory}

\subsection{Spherical functions}
Let $K_n:=U(n)$ denote the unitary group. The Heisenberg motion group $G_n:=H^n\rtimes U(n)$ with group law of $G_n$ given by
$$(z_1,t_1,k_1)(z_2,t_2,k_2)=(z_1+k_1z_2, t_1+t_2+\frac12Im (z_1\cdot\overline{k_1z_2}),k_1k_2).$$
The inverse of an element $(z,t,k)\in G_n$ is given by $(-k^{-1},-t,k^{-1})$ and $(0,0,I)$ serves as the identity where $I$ is the $n\times n$ identity matrix. Then $K_n$ and $H^n$ are isomorphic to two subgroups of $G_n$. 

A function $f$ on $H^{n}$ is said to be radial if $f$ is a function of $|z|$ and $t$, equivalently $f(kz,t)=f(z,t)$ for all $(z,t)\in H^n$ and $k\in K_n$. Let $L^1(H^n,K_n)$ denote the subspace of radial functions in $L^1(H^n)$. This space is canonically identical with $L^1(G_n,K_n)$, the  subspace of bi-$K_n$-invariant functions in $L^1(G_n)$. The radial functions on $H^n$ form a commutative convolution algebra, since this algebra is canonically isomorphic to the algebra of bi-$K_n$-invariant functions on $G_n$ and it is well-known (see e.g. \cite{FaHa87}, \cite{Str91}) that $(G_n,K_n)$ form a Gelfand pair.

Let $M_n:=M(H^n,K_n)$ denote the norm-closed convolution algebra generated by the surface measures $\{\sigma_{r,t}\}_{r>0,t\in\mathbb{R}}$ in $M(H^n)$. $M_n$ is canonically isomorphic to the algebra generated by $\{\sigma'_{r,t}\}_{r>0,t\in\mathbb{R}}$ on the corresponding spheres in $M(G_n)$.

As it is well-known (see e.g. \cite{BJR92}), the complex homomorphisms of a Gelfand pair $L^1(G_n,K_n)$ are given by bounded spherical functions. Bounded $K_n$-spherical functions on $G_n$ are characterized by satisfying $\phi(e)=1$ and the integrals equation
$$\int_K\phi(akb)dk=\phi(a)\phi(b),\;a,b\in G.$$
There are two families of spherical functions. One family is parametrised by $\lambda\in\mathbb{R}$, $\lambda\neq0$ and $k\in\mathbb{N}$. They are given by
$$e^{\lambda}_k(z,t)=\frac{k!(n-1)!}{(k+n-1)!}e^{i\lambda t}\varphi^\lambda_k(z),$$
where
$$\varphi^\lambda_k(z)=L^{n-1}_k(\frac{1}{2}|\lambda||z|^2)e^{-\frac{1}{4}\lambda||z|^2},$$
$L^{n-1}_k(t)$ being Laguerre polynomials of type $n-1$. Another family is given for each $u>0$ by
$$\eta_u(z,t)=\frac{2^{n-1}(n-1)!}{(u|z|)^{n-1}}J_{n-1}(u|z|)$$
where $J_{n-1}$ is the Bessel function of order $n-1$, and for $u=0$, the spherical function $\eta_0(z,t)=1$ identically.

Thus the Gelfand spectrum $\Sigma$ of the algebra $L^1(G_n,K_n)$, equivalently the algebra $L^1(H^n,K_n)$, is the union of the Laguerre spectrum $\Sigma_L=\mathbb{R}\setminus \{0\}\times\mathbb{N}$, the Bessel spectrum $\Sigma_B=(0,\infty)$ and the trivial character $\{0\}$. For any $\zeta\in\Sigma$, let $\varphi_\zeta$ be the spherical function $e^\lambda_k$, or $\eta_u$, or $\eta_0=1$ as the case may be. In what follows $\varphi_\zeta(r)$ stands for $\varphi_\zeta(z,0)$ with $|z|=r$.

Let us firstly collect some facts on the spectral theory of $M_n$ in the following lemma, which has been shown in \cite{NeTh97} and will be used in this paper.

\begin{lemma}\label{lem:spectral theory}
{\rm (i)}. $\sigma_{r_1,t_1}\ast\sigma_{r_2,t_2}$ is a probability measure absolutely continuous with respect to Haar measure on $H^n$, unless $n=1$, where the convolution of at least three spherical measures is absolutely continuous.

{\rm (ii)}. Finite linear combinations of functions of the form $\sigma_{r_1}\ast\sigma_{r_2}\ast\dotsm\ast\sigma_{r_k}$ where $r_i>0$ and $k\geq3$ are dense in $L^1(H^n,K_n)$, and thus $L^1(H^n,K_n)\subset M_n$.

{\rm (iii)}. The group generated by $S^n_r$ in $H^n$ equals $H^n$.

{\rm (iv)}. Let $\psi$ be a non-zero continuous complex homomorphism of $M_n$. Then the radial function $\varphi(g)=\psi(m_{K_n}\ast\delta_g\ast m_{K_n})$ equals $\varphi_\zeta(g)$ for some $\zeta\in\Sigma$. Whence, restriction of complex characters from $M(H^n,K_n)=M_n$ to its subalgebra $L^1(H^n,K_n)$ induces a canonical identification of the Gelfand spectrum of the two algebras.
\end{lemma}

 The fundamental property is the one stated in (i), which has previously been shown by Stempak in \cite{Ste85}. The property (ii) is deduced from (i) using approximation identity. Property (iii) follows easily from (i) by noting the support of $\sigma^{\ast6}_r$ contains an open neighbourhood of the identity. Property (iv) follows from (ii) combining some calculations on the spherical functions. We refer the reader to \cite{NeTh97} for more details.

Lemma \ref{lem:spectral theory} enables us to do the following arguments. Being a weak $*$-continuous action of $H^n$ on a von Neumann algebra $\M$, $\alpha$ induces a strongly continuous unitary representation of $H^n$ on $L_2(\M)$.  Then $\alpha$ determines canonically a norm continuous $*$-representation of the algebra $M(H^n,K_n)$. Let us denote by $\mathcal{A}_{\alpha}$ the commutative $C^*$-algebra which is the closure of $\alpha(M_n)$ in the operator norm. Let $\Sigma_\alpha$ denote the spectrum of $\mathcal{A}_\alpha$ which is by definition the set of all non-zero norm continuous complex homomorphisms of $\mathcal{A}_\alpha$.
Clearly $\Sigma_\alpha$ is a subset of the Gelfand spectrum of $M(H^n,K_n)$. Consequently, every symmetric (self-adjoint) measure $\mu$ in $M(H^n,K_n)$ is mapped to a self-adjoint operator on $L_2(\M)$, whose spectrum is the set $\{\varphi_\zeta(\mu):\;\zeta\in\Sigma_{\alpha}\}$. As a consequence, we have the following spectral decomposition,
\begin{align}\label{spectral decomposition}
\alpha(\sigma_r)=\int_{\zeta\in\Sigma_\alpha}\varphi_\zeta(r)de_{\zeta},
\end{align}
which is the starting point for our proof of the radial ergodic theorems using spectral methods.

\section{Mean ergodic theorem}
The mean ergodic theorem has been essentially proved in \cite{NeTh97}. But we prefer to recall it for completeness. On the other hand, some new terminologies appear. Let $(\M,H^n,\alpha)$ be a $W$*-dynamical system.
Let $\mathcal{F}=\{x\in\M:\;\alpha(g)x=x\;,\forall g\in G\}$. It is easy to show that $\mathcal{F}$ is a von Neumann subalgebra of $\M$ as in the semigroup case \cite{JuXu06}, thus there exists a unique conditional expectation $F:\;\M\rightarrow\mathcal{F}$. Moreover, this condition expectation extends naturally from $L_p(\M)$ to $L_p(\mathcal{F})$ with $1\leq p<\infty$, which are still denoted by $F$.

\begin{theorem}\label{thm:mean ergodic theorem}
For $x\in L_p(\M)$ with $1\leq p<\infty$
, we have
\begin{align}\label{mean ergodic theorem t8}
\|\alpha(\sigma_r)x-F(x)\|_p\rightarrow0, \;\mathrm{as}\;r\rightarrow\infty
\end{align}
and
\begin{align}\label{mean ergodic theorem t0}
\|\alpha(\sigma_r)x-x\|_p\rightarrow0, \;\mathrm{as}\;r\rightarrow0
\end{align}
\end{theorem}

\begin{proof}
Proof of (\ref{mean ergodic theorem t8}).
It suffices to prove the case $p=2$ since other cases can be obtained by the fact that $F$ and all $\alpha(\sigma_r)$'s are contractive and the standard density arguments. 
By (\ref{spectral decomposition}), for $x\in L_2(\M)$, we have
$$\|\alpha(\sigma_r)x\|^2_2=\int_{\Sigma_\alpha}|\varphi_\zeta(r)|^2\langle de_\zeta x,x\rangle.$$
Following from the asymptotic properties of the Laguerre and Bessel functions, all the nontrivial spherical functions satisfy $|\varphi_\zeta(r)|\leq1$ and $|\varphi_\zeta(r)|\rightarrow0$ as $r\rightarrow\infty$. Therefore if $x$ is an element such that its  spectral measure $\langle de_\zeta x,x\rangle$ assigns 0 to the trivial character then it follows by dominated convergence theorem that $\|\alpha(\sigma_r)x\|_2\rightarrow 0$ as $r\rightarrow\infty$.

Now we claim that for any $x\in L_2(\M)$, the spectral measure of $x-F(x)$ assigns 0 to the trivial character whence (\ref{mean ergodic theorem t8}) follows. This is deduced from the fact that the space of vectors invariant under each $\alpha(\sigma_r)$ coincides with $L_2(\mathcal{F})$, the space of $H^n$ invariant vectors. To see this assume $y$ is a unit vector such that $\alpha(\sigma_r)y=y$ for all $r>0$. The strict convexity of $H^n$ implies that $\alpha(g)y=y$ for almost all points on the support of $\sigma_r$. By strong continuity of $g\rightarrow\alpha_gy$, we have $\alpha(g)y=y$ for every point in the support of $\sigma_r$, which is the sphere $S^n_{r}$ in $H^n$. By Lemma \ref{lem:spectral theory} (iii), the group generated by $S^n_r$ equals $H^n$. It follows that $\alpha(g)y=y$ for all $g\in H^n$.

Proof of (\ref{mean ergodic theorem t0}).  Let $(\phi_{\varepsilon})_{\varepsilon>0}\subset{C_c^{\infty}(H^n)}$ be an approximation identity. By the strong continuity of the action $\alpha$, it is easy to verify that $\alpha(\phi_{\varepsilon})x\rightarrow x$ in $L_p$-norm as $\varepsilon\rightarrow 0$ for all $1\leq p<\infty$. Hence by density, it suffices to prove (\ref{mean ergodic theorem t0}) for $x=\alpha(\phi_{\varepsilon_0})y$ for some $y\in L_p(\M)$ and sufficiently small $\varepsilon_0$. By Minkowski inequality,
$$\|\alpha(\sigma_r)\alpha({\phi_{\varepsilon_{0}}})y-\alpha({\phi_{\varepsilon_{0}}})y\|_p\leq \|\sigma_r\ast\phi_{\varepsilon_0}-\phi_{\varepsilon_0}\|_1\|y\|_p,$$
and the desired result follows from the easy estimate
\begin{align*}
\|\sigma_r\ast\phi_{\varepsilon_0}-\phi_{\varepsilon_0}\|_1&=\|\int_{S^n_1}(\phi_{\varepsilon_0}(rh^{-1}g)-\phi_{\varepsilon_0}(g))d\sigma_1(h)\|_1\\
&\leq \int_{S^n_1}\|\phi_{\varepsilon_0}(rh^{-1}g)-\phi_{\varepsilon_0}(g)\|_1d\sigma_1(h)\leq C_{\phi_{\varepsilon_0}}r.
\end{align*}

\end{proof}

\section{Maximal ergodic theorems I}
This section is devoted to the proof of Theorem \ref{thm:maximal radial average}. Let $(\M, \tau, H^n,\alpha)$ be a $W$*-dynamical system. The main idea is to firstly embedd $\alpha(\sigma_r)x$ for $x$ in some nice dense subspace of $L_1(\mathcal M)\cap\mathcal M$ into an analytic family of operator-valued functions $M^a\alpha(\sigma_r)x$ ($a\in\mathbb C$) with $M^0\alpha(\sigma_r)x=\alpha(\sigma_r)x$, and then use Stein's analytic interpolation theorem.

To find the nice dense subspace, we need to introduce a space of $C^{\infty}$ vector associated to the action $\alpha$ of $H^n$, which is norm dense in the noncommutative $L_p$ space.
For the $W^*$-dynamical system $(\M,\tau,H^n,\alpha)$ with $\M$  and $H^n$ separable, by Theorem 5.5 of \cite{Ped79}, there exist a $C^*$-dynamical system $(\mathcal{A},\tau,H^n,\alpha)$ such that $\M=\mathcal{A}''$. The nice dense subspace is defined as
$$\mathcal{D}=\left\{\int_{H^n}\phi(g)\alpha_gxdg:\;\phi\in C^{\infty}_c(H^n),\;x\in\mathcal{A}\right\}.$$
The space $\mathcal{D}$ is contained in the $C^{\infty}$-class $\mathcal{A}^{\infty}$ in the terminology from noncommutative geometry  (see e.g. \cite{Con80}), since for any  $x=\int_{H^n}\phi(g)\alpha_gydg$ with some $y\in\mathcal{A}$ and $\phi\in C^{\infty}_c(H^n)$, the function
$$h\rightarrow \alpha_hx=\int_{H^n}\phi(h^{-1}g)\alpha_{g}ydg$$
is a $C^{\infty}$ map from $H^n$ to the norm space $\mathcal A$.  Moreover, for such a fixed $x\in\mathcal{D}$, the map given by
$$r\rightarrow\alpha(\sigma_r)x=\int_{H^n}(\phi\ast\sigma_r)\alpha_gy$$
is a $C^{\infty}$ map from $\mathbb{R}_+$ to the norm space $\mathcal A$, namely $(1/\delta)(\alpha(\sigma_{r+\delta})x-\alpha(\sigma_r)x)$ converges in $\|\cdot\|_\infty$ to an operator denoted by $(d/dr)\alpha(\sigma_r)x$ and similarly $(d^m/dr^m)\alpha(\sigma_r)x$ with $m\geq1$ exist as an operator in $\mathcal A$ .

Now we embedd $\alpha(\sigma_r)x$ for $x\in\mathcal{D}$ into an analytic family of operators using the Riemann-Liouville fractional integral operators. More precisely, for $x\in\mathcal{D}$, we consider the smooth operator-valued function on $\mathbb{R}_+$: $F_x(r)=\alpha(\sigma_r)x$. For $a\in\mathbb{C}$ with $Re(a)>0$, the Riemann-Liouville fractional integral of $F_x(r)$ are defined by
$$I^{a}F_x(r)=\frac{1}{\Gamma(a)}\int^r_0(r-s)^{a-1}F_x(s)ds.$$
The normalized fractional integrals are defined by
$$M^{a}F_x(r)=r^{-a}I^{a}F_x(r).$$
These operator-valued fractional integrals have an analytic continuation to $a\in\mathbb{C}$ (see e.g. the lemma in Page 77 of \cite{Ste70}). The family of operators $I^{a}$ satisfy the functional equation $I^{a}I^{b}F_{x}=I^{a+b}F_x$ and $I^0F_x=F_x$. Clearly, by definition $I^1F_x(r)=\int^r_0F_x(s)ds$. It follows that
$$I^{-1}F_x(r)=\frac{d}{dr}F_x(r)$$
and more generally for $k\in\mathbb{N}$
$$M^{-k}F_x(r)=r^k\frac{d^k}{dr^k}F_x(r.)$$

After these preparations, Theorem \ref{thm:maximal radial average} follows from the following two propositions.

\begin{proposition}\label{pro:maximal uniform average map}
Let $1<p\leq\infty$ and $n\geq1$. For $a\in\mathbb{C}$ with $Re(a)\geq1$, we have
\begin{align}\label{maximal uniform average map}
\|{\sup_{r>0}}^+M^{a}F_x(r)\|_p\leq C_{p}e^{\pi|Im(a)|}\|x\|_p,\;\forall x\in L_p(\M).
\end{align}
\end{proposition}

\begin{proposition}\label{pro:maximal uniform average ma2}
Let $n>1$ be the dimension of the Heisenberg group and $a\in\mathbb{C}$. The following inequality holds
\begin{align}\label{maximal uniform average ma2}
\|{\sup_{r>0}}^+M^{a}F_x(r)\|_2\leq Ce^{\pi|Im(a)|}\|x\|_2,\;\forall x\in\mathcal{D}
\end{align}
provided $-n+1+1/2<Re(a)<-n+2$.
\end{proposition}

The two propositions will be shown in the following two subsections. Now we can finish the proof of Theorem \ref{thm:maximal radial average} using Stein's analytic interpolation theorem. Here we present a more general result which implies Theorem \ref{thm:maximal radial average} when $a=0$.

\begin{theorem}\label{thm:maximal radial average mpc}
Let $n>1$ and  $a\in\mathbb{C}$ such that $-n+1+1/2<Re(a)<1$. Then for any $(2n-1)/(2n-2+Re(a))<p\leq\infty$, we have
\begin{align}\label{maximal radial average mpc}
\|{\sup_{r>0}}^+M^{a}F_x(r)\|_p\leq C_{p,a}\|x\|_p,\;\forall x\in L_p(\M).
\end{align}
\end{theorem}

\begin{proof}
It suffices to prove this result in the case $p\leq2$, since the results in the case $p>2$ can be obtained by complex interpolation from the result in the case $p=2$ and the trivial result in the case $p=\infty$. For $\theta\in[0,1]$ and $q\in[1,2]$, define the index $p$ as a function $p(\theta,q)$ on $[0,1]\times[1,2]$ satisfying the following relationship
$$\frac{1}{p(\theta,q)}=\frac{1-\theta}{2}+\frac{\theta}{q}.$$
It is easy to check that for fixed $\theta\in(0,1)$, $p(\theta,q)$ is a strictly increasing function associated to $q$ and  $p(\theta,(1,2])=(p(\theta,1),2]$.
For fixed $a\in\mathbb{C}$ with $-n+1+1/2<Re(a)<1$, we define $\theta_a$ such that
$$Re(a)=(1-\theta_a)(-n+1+1/2)+\theta_a$$
and $p_a=p(\theta_a,1)$.
It is easy to check that $p_a=(2n-1)/(2n-2+Re(a))$.

Let $p$ such that $p_a<p\leq2$. Then we can find $\theta$ smaller than but close to $\theta_a$, $q\in(1,2]$,  $b\in\mathbb{R}$ with $-n+1+1/2<b<\min(-n+2,Re(a))$, and $c\in\mathbb{R}$ with $c\geq1$ such that
$$\frac{1}{p}=\frac{1-\theta}{2}+\frac{\theta}{q},\;Re(a)=(1-\theta)\cdot b+\theta\cdot c.$$

Let $x\in L_p(\M)$ and $y=(y_r)_{r>0}$ with $\|x\|_p<1$ and $\|(y_r)_r\|_{L_{p'}(\ell_1)}<1$. Define
$$f(z)=u|x|^{\frac{p(1-z)}{2}+\frac{pz}{q}},\;\forall z\in\mathbb{C}$$
where $x=u|x|$ is the polar decomposition of $x$. On the other hand, by the complex interpolation of the space
$L_{p'}(\ell_1)$ (see e.g. Proposition 2.5 of \cite{JuXu06}). there is a function $g=(g_{r})_{r>0}$ continuous on the strip $\{z:\;0\leq Re(z)\leq1\}$ and analytic in the interior such that $g(\theta)=y$ and
$$\sup_{t\in\mathbb{R}}\max\{\|g(it)\|_{L_2(\ell_1)},\|g(1+it)\|_{L_{q'}(\ell_1)}\}<1.$$
Now define
$$G(z)=\exp(\delta(z^2-\theta^2))\sum_{r}\tau(M^{(1-z)b+zc+iIm(a)}F_{f(z)}(r)g_r(z))$$
where $\delta>0$ is a constant to be specified. $G$ is a function analytic in the open strip. Applying Proposition \ref{pro:maximal uniform average ma2},
\begin{align*}
|G(it)|&\leq\exp(\delta(-t^2-\theta^2))\|(M^{b+i(tc-tb+Im(a))}F_{f(it)}(r))_r\|_{L_2(\ell_{\infty})}\|g(it)\|_{L_2(\ell_1)}\\
&\leq C_{a}\exp(\delta(-t^2-\theta^2))\exp(\pi|tc-tb+Im(a))\|f(it)\|_{2}\\
&\leq C_{a}\exp(\delta(-t^2-\theta^2)+\pi|tc-tb+Im(a)|).
\end{align*}
Similarly, by Proposition \ref{pro:maximal uniform average map},
we deduce that
\begin{align*}
|G(1+it)|&\leq\exp(\delta(-(1+it)^2-\theta^2))\|(M^{c+i(tc-tb+Im(a))}F_{f(1+it)}(r))_r\|_{L_q(\ell_{\infty})}\\
&\;\;\;\;\;\;\;\;\;\;\;\times \|g(1+it)\|_{L_{q'}(\ell_1)}\\
&\leq C_{a,q}\exp(\delta(1-t^2-\theta^2))\exp(\pi|tc-tb+Im(a)|)\|f(1+it)\|_{q}\\
&\leq C_{a,q}\exp(\delta(1-t^2-\theta^2)+\pi|tc-tb+Im(a)|).
\end{align*}
Taking $\delta$ big enough, we have
$$\sup_{t\in\mathbb{R}}\max\{|G(it)|,|G(1+it)|\}<C_{a,p}.$$
Consequently, by the maximum principle, $|G(\theta)|<C_{a,p}$, that is
$$\Big|\sum_r\tau(M^aF_x(r)y_r)\Big|<C_{a,p},$$
which implies the desired result.
\end{proof}

\subsection{Proof of Proposition \ref{pro:maximal uniform average map}}
Inequality \eqref{maximal uniform average map} follows from the particular case $a=1$---Theorem \ref{thm:maximal uniform average ur} below.

\begin{proof}
Without loss of generality, we can assume $x\in L^+_p(\M)$. Let $(y_r)_{r>0}\subset L^+_{p'}(\M)$ where $p'$ is the conjugate index of $p$. Use the standard estimates of the $\Gamma$-function (see e.g. Page 79 of \cite{Ste70}),
$$\Big|\frac{\Gamma(Re(a))}{\Gamma(a)}\Big|\leq C_a e^{\pi|Im(a)|},$$
by Theorem \ref{thm:maximal uniform average ur}, we have
\begin{align*}
|\tau(M^{a}F_x(r)y_r)|&=\Big|\frac{1}{r^a}\frac{1}{\Gamma(a)}\int^r_0(r-s)^{a-1}\tau(F_x(s)y_r)ds\Big|\\
&\leq\frac{1}{r^{Re(a)}}\frac{C_ae^{\pi|Im(a)|}}{{\Gamma(Re(a))}}\int^r_0(r-s)^{Re(a)-1}\tau(F_x(s)y_r)ds\\
&\leq{C_ae^{\pi|Im(a)|}}\frac{1}{r}\int^r_0\tau(F_x(s)y_r)ds={C_ae^{\pi|Im(a)|}}\tau(\alpha(u_r)x\cdot y_r).
\end{align*}
Therefore,
$$\sum_{r>0}|\tau(M^{a}F_x(r)y_r)|\leq{C_ae^{\pi|Im(a)|}}\sum_{r>0}\tau(\alpha(u_r)x\cdot y_r).$$
Taking supremum over all $(y_r)_r\subset L^+_{p'}(\M)$  such that $\|\sum_ry_r\|_{p'}\leq1$, and using the fact that each element in the unit ball of $L_p(\M;\ell_1)$ is a sum of eight positive elements in the same ball (see Proposition 2.1 in \cite{JuXu06}) and Theorem \ref{maximal uniform average ur}, we deduce the assertion.
\end{proof}

\begin{theorem}\label{thm:maximal uniform average ur}
Let $\mu_r=(1/r)\int^r_0\sigma_sds$ denote the uniform averages of the spherical means on $H^n$ with $n\geq1$. Then for $1<p\leq\infty$, we have
\begin{align}\label{maximal uniform average ur}
\|{\sup_{r>0}}^{+}\alpha(\mu_r)x\|_p\leq C_{p,n}\|x\|_p,\;\forall x\in L_p(\M).
\end{align}
\end{theorem}

For the proof, we need the noncommutative Dunford-Schwartz maximal ergodic theorem \cite{JuXu06}.

\begin{lemma}\label{lem:p estimate for Mt}
Let $(T_t)_{t>0}$ be a strong continuous semigroup of linear maps on $\M$. Each
$T_t$ for $t\geq0$ satisfies the following properties:
\begin{enumerate}[\rm (i)]
\item $T_t$ is a contraction on $\M$:
$\|T_tx\|_{\infty}\leq\|x\|_{\infty}$ for all $x\in\M$;
\item $T_t$ is positive: $T_tx\geq0$ if $x\geq0$;
\item $\tau\circ T_t\leq\tau$: $\tau(T_t(x))\leq\tau(x)$ for all $x\in
L_1(\M)\cap\M^+$.
\end{enumerate}
Let $1<p\leq\infty$, then we have
\begin{align}\label{p estimate for
Mt}\|{\sup_{t>0}}^+ M_t(x)\|_p\leq C_p\|x\|_p,\;\;\forall x\in
L_p(\M),\end{align}
where
$$M_t(x)=\frac{1}{t}\int^t_0T_sds.$$
\end{lemma}

The proof of Theorem \ref{thm:maximal uniform average ur} is given as follows.

\begin{proof}
Fix $x\in L_p(\M)$.
For an action $\alpha$ of $H^n$, we have
\begin{align*}
\alpha(\mu_r)x&=\frac1r\int^r_0\alpha(\sigma_s)xds\\
&=\frac1r\int^r_0\int_{S^n_s}\alpha(g)xd\sigma_s(g)ds\\
&=\int_{S^n_1}\frac1r\int^r_0\alpha(sv)xdsd\sigma_1(v).
\end{align*}
Hence by triangle inequality
$$\|{\sup_{r>0}}^{+}\alpha(\mu_r)x\|_p\leq\int_{S^n_1}\|{\sup_{r>0}}^{+}\frac1r\int^r_0\alpha(sv)xds\|_pd\sigma_1(v).$$
Fix $v\in S^n_1$,  $\frac1r\int^r_0\alpha(sv)ds$ are the ergodic averages $M_r$ of the semigroup $(T_s)_{s>0}$ defined as
$T_sx=\alpha(sv)x$. The desired result follows from Lemma \ref{lem:p estimate for Mt} since this semigroup satisfies (i)-(iii) in the lemma.
\end{proof}

\begin{remark}\label{rem:p estimate for Mt bar}
It is easy to observe that the same argument works also for the the uniform averages of the spherical means on $\bar{H}^n$. Thus inequality \eqref{maximal uniform average ur} remains true if we replace $\mu_r$ by $\bar{\mu}_r$ defined as
$$\bar{\mu}_r=(1/r)\int^r_0\bar{\sigma}_sds.$$
\end{remark}


\subsection{Proof of Proposition \ref{pro:maximal uniform average ma2}}
This proposition will be shown by two steps. The first step is the following noncommutative Littlewood-Paley $g$-function in $L_2(\M)$.

\begin{lemma}\label{lem:g function}
Let $n>1$. For any integer $1\leq m\leq n-1$ and $x\in\mathcal{D}$, define
$$g_m(x)=\Big(\int^{\infty}_0r^{2m-1}\big|\frac{d^m}{dr^m}\alpha(\sigma_r)x\big|^2dr\Big)^\frac 12.$$
Then we have $\|g_m(x)\|_2\leq C_m\|x\|_2.$
\end{lemma}

This lemma can be deduced from the following equality
$$\big\|\frac{d^m}{dr^m}\alpha(\sigma_r)x\big\|^2_2=\int_{\Sigma_\alpha}\big|\frac{d^m}{dr^m}\varphi_\zeta(r)\big|^2\langle de_\zeta x,x\rangle,$$
and the spectral estimates---Lemma 4.4 of \cite{NeTh97}: For $1\leq m\leq n-1$ with $n\geq2$, we have
\begin{align*}
\sup_{\zeta\in\Sigma}\int^{\infty}_0\Big|\frac{d^m}{dr^m}\varphi_{\zeta}(r)\Big|^2r^{2m-1}dr\leq C_m(n),
\end{align*}
where $C_m(n)$ depends only on the algebra $M(H^n,K)$. We omit the details. 

This lemma, together with Theorem \ref{thm:maximal uniform average ur}, implies the second step ---the following maximal inequality for fractional averages in the integer case.

\begin{lemma}\label{lem:maximal uniform average maz}
Let $n>1$. Then we have
\begin{align}\label{maximal uniform average maz}
\|{\sup_{r>0}}^+M^{-n+2}F_x(r)\|_2\leq C\|x\|_2,\;\forall x\in\mathcal{D}.
\end{align}
\end{lemma}

\begin{proof}
We first prove this result in the case $n=2$. Then the cases $n\geq3$ are proved using induction. Without loss of generality, we assume $x\in L^+_2(\M)\cap\mathcal{D}$.

The case $n=2$. As in \cite{Nev94} we try to compare $\alpha(\sigma_t)$ with $\alpha(\mu_r)$ in order to make use of the maximal inequality for $\alpha(\mu_r)$. Fix $y\in L_2(\M)$, and consider the differentiable scalar function $f_{x,y}(r)=\tau(\alpha(\sigma_r)(x)y^*)$. Using integration by parts, by (\ref{mean ergodic theorem t0}), we obtain the equality
$$\int^r_0sf'_{x,y}(s)ds=rf_{x,y}(r)-\int^r_0f_{x,y}(s)ds.$$
This equality is valid for all $y\in L_2(\M)$, hence it also holds in $L_2$ for the operator-valued function $F_x(r)=\alpha(\sigma_r)x$. Therefore, dividing by $r$, we get
$$\alpha(\sigma_r)x-\alpha(\mu_r)x=\frac{1}{r}\int^r_0s\frac{d}{ds}\alpha(\sigma_s)xds.$$
Now using the fact $x\leq|x|$ for a self-adjoint operator, we have
\begin{align*}
\alpha(\sigma_r)x-\alpha(\mu_r)x&\leq\left(\Big|\frac{1}{r}\int^r_0s\frac{d}{ds}\alpha(\sigma_s)xds\Big|^{2}\right)^{\frac12}\\
&\leq\left(\frac{1}{r^2}\int^r_0sds\int^r_0s\big|\frac{d}{ds}\alpha(\sigma_s)x\big|^2ds\right)^{\frac12}\\
&=\frac{1}{\sqrt{2}}\left(\int^r_0s\big|\frac{d}{ds}\alpha(\sigma_s)x\big|^2ds\right)^{\frac12}\leq \frac{1}{\sqrt{2}}g_1(x).
\end{align*}
We have used the Kadison-Schwarz inequality in the second inequality as well as  L\"owner-Heinz inequality in the second and the third inequality. Then we finish the proof in this case by triangle inequality in $L_2(\M;\ell_{\infty})$, Theorem \ref{thm:maximal uniform average ur} and Lemma \ref{lem:g function}.

The case $n\geq3$. Using integration by parts,
\begin{align*}
\int^r_0s^{n-1}\frac{d^{n-1}}{ds^{n-1}}F_x(s)ds&=s^{n-1}\frac{d^{n-2}}{ds^{n-2}}F_x(s)\big|^{r}_0-\int^r_0(n-1)s^{n-2}\frac{d^{n-2}}{ds^{n-2}}F_x(s)ds\\
&=r^{n-1}\frac{d^{n-2}}{dr^{n-2}}F_x(r)-\int^r_0(n-1)s^{n-2}\frac{d^{n-2}}{ds^{n-2}}F_x(s)ds.
\end{align*}
Dividing by $t$ on both sides, by Kadison-Schwarz inequality and L\"owner-Heinz inequality, we get
\begin{align*}
r^{n-2}\frac{d^{n-2}}{dr^{n-2}}F_x(r)&=\frac{1}{r}\int^r_0s^{n-1}\frac{d^{n-1}}{ds^{n-1}}F_x(s)ds+\frac{1}{r}\int^r_0(n-1)s^{n-2}\frac{d^{n-2}}{ds^{n-2}}F_x(s)ds\\
&\leq \left(\frac{1}{r^2}\int^r_0sds\int^r_0s^{2n-3}\big|\frac{d^{n-1}}{ds^{n-1}}\alpha(\sigma_s)x\big|^2ds\right)^{\frac12}\\
&\;\;\;\;\;\;\;\;\;\;\;\;\;+(n-1)\left(\frac{1}{r^2}\int^r_0sds\int^r_0s^{2n-5}\big|\frac{d^{n-2}}{ds^{n-2}}\alpha(\sigma_s)x\big|^2ds\right)^{\frac12}\\
&\leq \frac{1}{\sqrt{2}}g_{n-1}(x)+\frac{1}{\sqrt{2}}(n-1)g_{n-2}(x).
\end{align*}
Finally, Lemma \ref{lem:g function} implies the desired result by the triangle inequality in $L_{2}(\M;\ell_{\infty})$.
\end{proof}

\begin{remark}
The arguments in the case $n=2$ have already provided the proof of the spherical maximal inequality in $L_2$ for dimension $n\geq2$,
\begin{align*}
\|{\sup_{r>0}}^+\alpha(\sigma_r)x\|_2\leq C\|x\|_2,\;\forall x\in L_2(\mathcal M).
\end{align*}
\end{remark}

The following result has been obtained in  Lemma 5.7 of \cite{JuXu06}.
\begin{lemma}\label{lem:elementary}
Let $x=(x_s)_s\in L_p(\M;\ell_\infty)$ and $(z_{r,s})_{r,s}\subset\mathbb{C}$. Then
\begin{align*}
\left\|\Big(\int^\infty_0 z_{r,s}x_sds\Big)_{r}\right\|_{L_p(\ell_{\infty})}\leq \sup_r\Big(\int^\infty_0|z_{r,s}|ds\Big)\|(x_s)_{s}\|_{L_p(\ell_{\infty})}
\end{align*}
\end{lemma}

Now we are ready to give the proof of Proposition \ref{pro:maximal uniform average ma2}.

\begin{proof}
Assume $x\in L^+_2(\M)\cap\mathcal{D}$. Given $a\in\mathbb{C}$ with $-n+1+1/2<Re(a)<-n+2$, we can write $a=-\delta-n+2+i\gamma$ with $0<\delta=-n+2-Re(a)<1/2$. By semigroup property of the fractional integrals, we have
\begin{align*}
I^{-\delta+i\gamma-n+2}F_x(r)&=I^{-\delta+i\gamma}I^{-n+2}F_x(r)\\
&=\frac{1}{\Gamma(-\delta+i\gamma)}\int^{r/2}_0(r-s)^{-\delta-1+i\gamma}I^{-n+2}F_x(s)ds\\
&\;\;\;\;\;\;\;\;\;+\frac{1}{\Gamma(1-\delta+i\gamma)}(r/2)^{-\delta+i\gamma}I^{-n+2}F_x(r/2)\\
&\;\;\;\;\;\;\;\;\;+\frac{1}{\Gamma(1-\delta+i\gamma)}\int_{r/2}^r(r-s)^{-\delta+i\gamma}\frac{d}{ds}I^{-n+2}F_x(s)ds.
\end{align*}
Multiplying $r^{\delta-i\gamma+n-2}$ on both sides, we have
\begin{align*}
&M^{-\delta+i\gamma-n+2}F_x(r)\\
&=\frac{1}{\Gamma(-\delta+i\gamma)}\int^{r/2}_0r^{\delta-i\gamma}(r-s)^{-\delta-1+i\gamma}M^{-n+2}F_x(s)ds\\
&\;\;\;\;\;+\frac{1}{\Gamma(1-\delta+i\gamma)}(1/2)^{-\delta+i\gamma}M^{-n+2}F_x(r/2)\\
&\;\;\;\;\;+\frac{1}{\Gamma(1-\delta+i\gamma)}\int_{r/2}^r(r-s)^{-\delta+i\gamma}r^{\delta-1/2+i\gamma}r^{n-3/2}I^{-n+1}F_x(s)ds\\
&=I_r(x)+II_r(x)+III_r(x).
\end{align*}
The second term is trivially estimated by Lemma \ref{lem:maximal uniform average maz}, hence
$$\|(II_r(x))_{r>0}\|_{L_2(\ell_{\infty})}\leq c_{n,\delta}\|x\|_2.$$ The first term follows from Lemma \ref{lem:elementary} and \ref{lem:maximal uniform average maz}, since $-\delta+i\gamma$ is not a pole of $\Gamma(z)$, whence
\begin{align*}
&\frac{1}{|\Gamma(-\delta+i\gamma)|}\int^{r/2}_0|r^{\delta-i\gamma}(r-s)^{-\delta-1+i\gamma}|ds\\
&\leq c_{\gamma}\frac{1}{|\Gamma(-\delta)|}\int^{1/2}_0(1-s)^{-\delta-1}ds<\infty.
\end{align*}
Using the duality argument  as in the proof of Proposition \ref{pro:maximal uniform average map}, the $L_p(\ell_\infty)$-norm of $III_r(x)$ is controlled by the  $L_p(\ell_\infty)$-norm of the following sequence modulo a constant depending on $\gamma$,
$$\frac{1}{\Gamma(1-\delta)}\int_{r/2}^r(r-s)^{-\delta}r^{\delta-1/2}r^{n-3/2}|I^{-n+1}F_x(s)|ds.$$
By Kadison-Schwarz inequality and L\"owner-Heinz inequality, noting that $r\leq 2s$,  the previous quantity is controlled by the following expression
\begin{align*}
&\frac{1}{\Gamma(1-\delta)}\left(r^{2\delta-1}\int^r_{r/2}(r-s)^{-2\delta}ds\int^r_{r/2}r^{2n-3}\big|I^{-n+1}F_x(s)\big|^2ds\right)^{\frac12}\\
&\leq c_{\delta}\left(\int^r_{r/2}s^{2n-3}\big|I^{-n+1}F_x(s)\big|^2ds\right)^{\frac12}\leq c_{\delta}g_{n-1}(x).
\end{align*}
Hence by Lemma \ref{lem:g function}, $$\|(III_r(x))_{r>0}\|_{L_2(\ell_{\infty})}\leq c_{n,a}\|x\|_2.$$
\end{proof}

\section{Maximal ergodic theorems  II}
In this section, we show the noncommutative maximal ergodic theorem for the $W^*$-dynamical system $(\M,\tau,\bar{H}^n,\alpha)$. The following transference principle whose proof  will appear in the subsequent work \cite{Hon} reduces the maximal ergodic theorem for the abstract dynamical system to the specific one $(L_{\infty}(\bar{H}^n)\overline{\otimes}\mathcal{\mathcal{M}},\int\otimes\tau,\bar{H}^n,T\otimes id_{\mathcal{M}})$, where $T$ is the translation on $\bar{H}^n$.

\begin{lemma}\label{lem:transference}
Let $1<p\leq\infty$. If 
\begin{align*}
\|{\sup_{r\geq1}}^+f\ast \bar{\sigma}_r\|_p\leq C_{p}\|f\|_p,\;\forall f\in L_p(L_\infty(\bar{H}^n)\overline{\otimes}\M),
\end{align*}
then we have
\begin{align*}
\|{\sup_{r\geq1}}^+\alpha(\bar{\sigma}_r)x\|_p\leq C_{p}\|x\|_p,\;\forall x\in L_p(\M).
\end{align*}
\end{lemma}

For this particular dynamical system, every operator-valued function $f\in L_2(\M)$ can be rewritten as $f=f_B+f_L$. Here the spectral measure of $f_B$ is supported in $\Sigma_B$, while the spectral measure of $f_L$ is supported in $\Sigma_L$. Moreover, by \cite{HuRi80} (see also \cite{NeTh97}), we have an explicit formula of $f_B$,
$$f_B(z,t)=\mathcal{E}f(z,t)=\frac{1}{2\pi}\int^{2\pi}_0f(z,s)ds.$$
Clearly, this averaging operator $\mathcal{E}$ can be extended contractively to all $L_p(\M)$ with $1\leq p\leq\infty$. Hence the operator $I-\mathcal{E}$ is bounded on all $L_p(\M)$ with $1\leq p\leq\infty$. Therefore for each $f\in L_p(\M)$, we can write $f=f_B+f_L=\mathcal{E}f+(I-\mathcal{E})f$ with $f_B, f_L\in L_p(\M)$.

Let $n>1$ and $2n/(2n-1)<p\leq\infty$. Theorem  \ref{thm:maximal radial average reduced}  follows from the following two inequalities by triangle inequality:
\begin{align}\label{map fb}
\|{\sup_{r>0}}^+f_B\ast\bar{\sigma}_r\|_p\leq C_p\|f_B\|_p
\end{align}
and
\begin{align}\label{map fl}
\|{\sup_{r\geq1}}^+f_L\ast\bar{\sigma}_r\|_p\leq C_p\|f_L\|_p.
\end{align}
Since $f_B$ does not depend on the central variable, it can be identified with a function $\tilde{f}_B$ on the Euclidean space $\mathbb{R}^{2n}$. Furthermore, it is trivial that $f_B\ast\bar{\sigma}_r=\tilde{f}_B\ast\tilde{\sigma}_r$, where $\tilde{\sigma}_r$ are the sphere averages on $\mathbb{R}^{2n}$. Hence inequality (\ref{map fb}) follows from Proposition 4.1 in  \cite{Hon14}.

To deal with inequality (\ref{map fl}), we shall again embed  $f_L\ast\bar{\sigma}_r$ into an analytic family of operator-valued functions and using analytic interpolation arguments.
As in the case  $\mathcal{N}=\mathbb{C}$ in Page 327 of \cite{NeTh97}, we have the following explicit expression of $f\ast\bar{\sigma}_r$,
\begin{align}\label{explicite expression 2}
f\ast\bar{\sigma}_r(z,t)=\sum_{\lambda\in\mathbb{Z}}\left(\sum^\infty_{k=0}\psi^{n-1}_k(\sqrt{|\lambda|}r)f^{\lambda}\ast_{\lambda}\varphi^\lambda_k(z)\right)e^{-i\lambda t},
\end{align}
 for all $f=f_L$ belonging to $\mathscr{S}(H^n)\otimes S_{\mathcal{N}}$, which is dense in all $L_p(\M)$. In the sequel, a nice function always means a function in this class.
The notations are explained as follows. The expression (\ref{explicite expression 2}) can be deduced as follows. Note first that
$$f\ast\bar{\sigma}_r(z,t)=\sum_{\lambda\in\mathbb Z}(f\ast\bar{\sigma}_r)^\lambda(z)e^{-i\lambda t},$$
where $f^\lambda(z)$ stands for the partial inverse Fourier transform
\begin{align}\label{partial inverse fourier transform}
f^\lambda(z)=\frac{1}{2\pi}\int^{2\pi}_0f(z,t)e^{i\lambda t}dt.
\end{align}
Now an easy calculation shows that
$$(f\ast\bar{\sigma}_r)^\lambda(z)=f^\lambda\ast_\lambda\bar{\sigma}_r(z)$$
where $\ast_\lambda$ is the $\lambda$-twisted convolution for functions on $\mathbb{C}^n$ given by
\begin{align}\label{twisted convolution}
g\ast_\lambda h(z)=\int_{\mathbb{C}_n}g(z-\omega)h(\omega)e^{i\frac{\lambda}{2}Im(z\cdot \bar{\omega})}d\omega.
\end{align}
As seen from Theorem 4.1 of \cite{Tha91}, for operator-valued functions, we also have
$$f^\lambda\ast_\lambda\bar{\sigma}_r(z)=\sum^\infty_{k=0}\psi^{n-1}_k(\sqrt{|\lambda|}r)f^{\lambda}\ast_{\lambda}\varphi^\lambda_k(z)$$
where
$\varphi^\lambda_k(z)$ is the dilated Laguerre function
$$\varphi^\lambda_k(z)=L^{n-1}_k\Big(\frac{1}{2}|\lambda||z|^2\Big)e^{-\frac{1}{4}|\lambda||z|^2}$$
with $L^{n-1}_k$ being Laguerre polynomials of type $n-1$, and $\psi^{n-1}_k(r)$, or more generally, $\psi^a_k(r)$ is the normalized Laguerre function
$$\psi^a_k(r)=\frac{\Gamma(k+1)\Gamma(a+1)}{\Gamma(k+a+1)}L^a_k\Big(\frac12 r^2\Big)e^{-\frac14r^2}.$$

Note that the Laguerre functions $\psi^a_k$ are defined even for complex $a$ with $Re(a)>-1$. In this section,
we embed $\alpha(\sigma_r)$ into an analytic family of operator-valued functions using $\psi^a_k$ instead of using Riemann-Liouville fractional integrals. More precisely, we consider the following analytic family of operators
\begin{align}\label{analytic family}
\bar{M}^a_r f(z,t)=\sum_{\lambda\in\mathbb{Z}}\left(\sum^\infty_{k=0}\psi^{n-1+a}_k(\sqrt{|\lambda|}r)f^{\lambda}\ast_{\lambda}\varphi^\lambda_k(z)\right)e^{-i\lambda t}.
\end{align}
From this expression, $\bar{M}^a_rf$ with $Re(a)>-n+1-1/3$ is well-defined in $L_2(\M)$ for a nice function $f$, since it is known from Szeg\"o \cite{Sze67} that  a uniform estimate holds in this case, that is
\begin{align*}
\sup_{k\geq0}\sup_{r>0}|\psi^{n-1+a}_k(r)|\leq C
\end{align*}
for all $Re(a)>-n+1-1/3$.

With above preparations, we formulate the main result of this section is stated as follows.

\begin{theorem}\label{thm:maximal psi map 2}
Let $n>1$ and  $a\in\mathbb{C}$ such that $-n+1<Re(a)<1$. Then for any $2n/(2n-1+Re(a))<p\leq\infty$, we have
\begin{align}\label{maximal psi map 2}
\|{\sup_{r>0}}^+\bar{M}^{a}_rf\|_p\leq C_{p,a}\|f\|_p,\;\forall f=f_L\in L_p(\M).
\end{align}
\end{theorem}

In particular, when $a=0$, we obtain inequality (\ref{map fl}), thus finish the proof of Theorem  \ref{thm:maximal radial average reduced}. This theorem can be deduced from the same analytic interpolation argument used in the proof of Theorem \ref{thm:maximal radial average mpc} based on the following two propositions. The first one corresponds to Proposition \ref{thm:maximal uniform average ur}.

\begin{proposition}\label{pro:maximal psi ap}
Let $1<p\leq\infty$ and $n\geq1$. For $a\in\mathbb{C}$ with $Re(a)\geq1$, we have
\begin{align}\label{maximal psi ap}
\|{\sup_{r>0}}^+\bar{M}^{a}_rf\|_p\leq C_{p,a}\|f\|_p,\;\forall f=f_L\in L_p(\M).
\end{align}
\end{proposition}

The second one corresponds to Proposition \ref{pro:maximal uniform average ma2}. It is this proposition that provides the optimal scale of $p$ in Theorem \ref{thm:maximal psi map 2}.
\begin{proposition}\label{pro:maximal psi a2 2}
Let $n>1$ be the dimension of the reduced Heisenberg group and $a\in\mathbb{C}$. The following inequality holds
\begin{align}\label{maximal psi a2 2}
\|{\sup_{r\geq1}}^+\bar{M}^{a}_rf\|_2\leq C_{a}\|f\|_2,\;\forall f=f_L\in L_2(\M).
\end{align}
provided $Re(a)>-n+1$.
\end{proposition}

\subsection{Proof of Proposition  \ref{pro:maximal psi ap}}

In order to prove Proposition \ref{pro:maximal psi ap}, we rewrite $\bar{M}^a_rf$ for nice band-limited function $f$ in another form. For this, we require the following formula relating Laguerre polynomials of different type:
\begin{align}\label{laguerre polynomial}
L^a_k(r)=\frac{\Gamma(k+a+1)}{\Gamma(a-b)\Gamma(k+b+1)}\int^1_0s^b(1-s)^{a-b-1}L^b_k(rs)ds
\end{align}
holds true for $Re(a-b)>0$ and $Re(b)>-1$.
Defining
\begin{align*}
(P_rf)^\lambda(z)=e^{-\frac14|\lambda|r}f^\lambda(z)
\end{align*}
to be the Poisson integrals of $f$ in the $t$-variable. Then an easy calculation yields that
$$(\bar{M}^a_rf)^\lambda(z)=L(a,b)\int^1_0s^{2n+2b-1}(1-s^2)^{a-b-1}e^{-\frac14|\lambda|(r^2(1-s^2))}(\bar{M}^b_{rs}f)^\lambda(z)ds$$
where
$$L(a,b)=\frac{\Gamma(a+n)}{\Gamma(a-b)\Gamma(b+n)},$$
which in turn deduces that
\begin{align}\label{analytic family2}
\bar{M}^a_rf(z,t)=L(a,b)\int^1_0s^{2n+2b-1}(1-s^2)^{a-b-1}(P_{r^2(1-s^2)}\bar{M}^b_{rs}f)(z,t)ds.
\end{align}

Now, we are at a position to prove Proposition \ref{pro:maximal psi ap}.
\begin{proof}
Without loss of generality, we assume $f=f_L\in L^+_p(\M)$.
Given $a\in\mathbb{C}$ with $Re(a)\geq1$, we rewrite it as $a=1+i\gamma+\delta$ with $\delta\geq0$. Since for any $r_1,r_2>0,$ $Re(b)>-n+1-1/3$,
\begin{align}\label{commute}
P_{r_1}\bar{M}^b_{r_2}f=\bar{M}^b_{r_2}P_{r_1}f,
\end{align}
by (\ref{analytic family2}), $\bar{M}^{1+i\gamma+\delta}_rf(z,t)$ can be expressed as
\begin{align*}
L(1+i\gamma+\delta,0)\int^1_0s^{2n-1}(1-s^2)^{i\gamma+\delta}(P_{r^2(1-s^2)}f)\ast\bar{\sigma}_{rs}(z,t)ds.
\end{align*}
Now applying a duality argument used in the proof of Proposition \ref{pro:maximal uniform average map}, we obtain that the $L_p(\M;\ell_\infty)$-norm of $\bar{M}^{1+i\gamma+\delta}_rf$ is controlled modulo a constant depending on $a$ by the
$L_p(\M;\ell_\infty)$-norm of
\begin{align*}
L(1+\delta,0)\int^1_0s^{2n-1}(1-s^2)^{\delta}(P_{r^2(1-s^2)}f)\ast\bar{\sigma}_{rs}(z,t)ds.
\end{align*}
Using Lemma \ref{lem:p estimate for Mt} for the Poisson semigroup in the $t$-variable on $\M$, we can find $F\in L^+_p(\M)$ such that
$$P_{r^2(1-s^2)}f\leq F,\;\forall r,s\;\mathrm{and}\;\|F\|_p\leq C_p\|f\|_p.$$
Hence the $L_p(\M;\ell_\infty)$-norm of $\bar{M}^{1+i\gamma+\delta}_rf$ is controlled modulo a constant by the
$L_p(\M;\ell_\infty)$-norm of
\begin{align*}
&L(1+\delta,0)\int^1_0s^{2n-1}(1-s^2)^{\delta}F\ast\bar{\sigma}_{rs}(z,t)ds\\
&\leq L(1+\delta,0)\frac{1}{r^{2n}}\int^r_0s^{2n-1}F\ast\bar{\sigma}_{s}(z,t)ds\\
&\leq L(1+\delta,0)\frac{1}{r}\int^r_0F\ast\bar{\sigma}_{s}(z,t)ds,
\end{align*}
which, by Remark \ref{rem:p estimate for Mt bar}, is  smaller than $L_p(\M)$-norm of $F$, whence $L_p(\M)$-norm of $f$ modulo a constant. Thus we finish the proof.
\end{proof}

\subsection{Proof of Proposition \ref{pro:maximal psi a2 2}}

The proof of Proposition \ref{pro:maximal psi a2 2} is much more involved. The following argument is an improvement of the one in the commutative case, i.e. Proposition 6.2 of \cite{NeTh97}, where they only give a proof in the real number case which apparently is not enough to apply the method of analytic interpolation. The author would like to thank Sundaram Thangavelu for helpful discussion to make a remedy. On the other hand, we need the following noncommutative version of Lemma 1 in Page 499 of \cite{Ste93}, which may be useful somewhere else.

\begin{lemma}\label{lem:pointwise control}
Suppose $F$ is a $S_\mathcal{N}$-valued function on $\mathbb{R}$, which smooth for $t$ in an interval $I$. Then for each $\ell$ with $\ell\leq|I|$, we have
\begin{align}\label{pointwise control}
|F(t)|^2\leq 2\ell^{-1}\int_I|F(s)|^2ds+2\ell\int_I|F'(s)|^2ds
\end{align}
for each $t\in I$.
\end{lemma}

\begin{proof}
Each $t\in I$ belongs to an interval $I_0\subset I$ with $|I_0|=\ell$. Observe that
$$F(t)-F(s)=\int^t_sF'(u)du$$
which implies
$$F(t)-\frac{1}{|I_0|}\int_{I_0}F(s)ds=\frac{1}{|I_0|}\int_{I_0}(F(t)-F(s))ds=\frac{1}{|I_0|}\int_{I_0}\int^t_sF'(u)duds.$$
Hence by Kadison-Schwarz's inequalities
\begin{align*}
|F(t)|^2&\leq 2\left|\frac{1}{|I_0|}\int_{I_0}F(s)ds\right|^2+2\left|\frac{1}{|I_0|}\int_{I_0}\int^t_sF'(u)duds\right|^2\\
&\leq2\frac{1}{|I_0|}\int_{I_0}|F(s)|^2ds+2\frac{1}{|I_0|}\int_{I_0}\left|\int^t_sF'(u)du\right|^2ds\\
&\leq2\frac{1}{|I_0|}\int_{I_0}|F(s)|^2ds+2\frac{1}{|I_0|}\int_{I_0}|t-s|\int^t_s|F'(u)|^2duds\\
&\leq2\ell^{-1}\int_{I}|F(s)|^2ds+2\ell\int_{I}|F'(u)|^2du.
\end{align*}
\end{proof}

Now let us give the proof of Proposition \ref{pro:maximal psi a2 2}.

\begin{proof}
We decompose $\Sigma_L$  as follows. For $j=0,1,2,\dotsm$, define
$$\Sigma_j=\{(\lambda,k):\;|\lambda|k\in[2^j,2^{j+1})\}.$$
Let $\Sigma_{-1}=\{(\lambda,0):\;\lambda\neq0\}$. Then we have $\Sigma_L=\cup_{j\geq-1}\Sigma_j$.

Let $f=f_L$ be a nice function. Write $a=-n+1+\delta+i\gamma$. Let $E_j$ be the projection operator corresponding to the set $\Sigma_j$. By triangle inequality, it suffice to show that
\begin{align*}
\|{\sup_{r\geq1}}^+\bar{M}^{a}_rE_jf\|_2\leq C_{a,\delta'} 2^{-\delta' j/2}\|f\|_2
\end{align*}
for some $0<\delta'<\delta$. 
By Lemma \ref{lem:l8c+l8r}, we only need to prove
\begin{align*}
\|{\sup_{r\geq1}}^+|\bar{M}^{a}_rE_jf|^2\|^{1/2}_1\leq C_{a,\delta'} 2^{-\delta' j/2}\|f\|_2.
\end{align*}
Now applying Lemma \ref{lem:pointwise control}, for every $L>0$, $r\geq1$, we have
\begin{align*}
|\bar{M}^{a}_rE_jf|^2\leq 2L^{-1}\int^\infty_1|\bar{M}^{a}_rE_jf|^2dr+2L\int_1^\infty\Big|\frac{d}{dr}\bar{M}^{a}_rE_jf\Big|^2dr.
\end{align*}
Therefore, it suffices to prove that for $(\lambda,k)\in\Sigma_j$,
$$L^{-1}\int^\infty_1|\psi^{\delta+i\gamma}_k(\sqrt{|\lambda|}r)|^2dr+L\int_1^\infty\Big|\frac{d}{dr}(\psi^{{\delta+i\gamma}}_k(\sqrt{|\lambda|}r))\Big|^2dr\leq C_{a,\delta'}2^{-\delta' j}.$$
Note that when $j=-1$, $\psi^a_0(\sqrt{|\lambda|}r)=e^{-\lambda r^2/4}$. We trivially have the desired estimate. When $j\geq0$, by the formula (\ref{laguerre polynomial}), an easy  calculation yields
$$\psi^{\delta+i\gamma}_k(\sqrt{|\lambda|}r)=C(a,\delta')\int^1_0s^{\delta'}(1-s)^{\delta-\delta'+i\gamma-1}e^{\frac14|\lambda|r^2(s-1)}\psi^{\delta'}_k(\sqrt{|\lambda|s}r)ds$$
where
$$C(a,\delta')=\frac{\Gamma(\delta+i\gamma+1)}{\Gamma(\delta-\delta'+i\gamma)\Gamma(\delta'+1)}.$$
Whence $\frac{d}{dr}\psi^{\delta+i\gamma}_k(\sqrt{|\lambda|}r)$ equals
\begin{align*}
&C(a,\delta')\int^1_0s^{\delta'}(1-s)^{\delta-\delta'+i\gamma-1}(\frac{1}{2}|\lambda|r(s-1))e^{\frac14|\lambda|r^2(s-1)}\psi^{\delta'}_k(\sqrt{|\lambda|s}r)ds\\
&+\;\;\;C(a,\delta')\int^1_0s^{\delta'}(1-s)^{\delta-\delta'+i\gamma-1}e^{\frac14|\lambda|r^2(s-1)}\frac{d}{dr}(\psi^{\delta'}_k(\sqrt{|\lambda|s}r))ds.
\end{align*}
For complex number $b$, denote
$$A(b,k,\lambda)=\left(\int^\infty_1|\psi^{b}_k(\sqrt{|\lambda|}r)|^2dr\right)^{1/2}$$
and
$$B(b,k,\lambda)=\left(\int_1^\infty\Big|\frac{d}{dr}(\psi^{{b}}_k(\sqrt{|\lambda|}r))\Big|^2dr\right)^{1/2}.$$
We claim that the following two estimates hold for any $\eta>0$, which will be shown later:
\begin{align}\label{fundamental estimate 1}
A^2(\delta',k,\eta)\leq C(\eta k)^{-\delta'-1/2}
\end{align}
and
\begin{align}\label{fundamental estimate 2}
B^2(\delta',k,\eta)\leq C(\eta k)^{-\delta'+1/2}.
\end{align}
Therefore, by Minkowski inequality, we have
\begin{align*}
A(\delta+i\gamma,k,\lambda)&\leq C_{a,\delta'}\int^1_0s^{\delta'}(1-s)^{\delta-\delta'-1}A(\delta',k,\lambda s)ds\\
&\leq C_{a,\delta'}\int^1_0s^{\delta'}(1-s)^{\delta-\delta'-1}(|\lambda|s k)^{-\delta'/2-1/4}ds\\
&=C_{a,\delta'}(|\lambda|k)^{-\delta'/2-1/4}\int^1_0s^{\delta'/2-1/4}(1-s)^{\delta-\delta'-1}ds\\
&\leq C_{a,\delta'}(|\lambda|k)^{-\delta'/2-1/4}.
\end{align*}
Together with the fact $ue^{-u}\leq C$ for any $u>0$, Minkowski inequality implies that
\begin{align*}
B(\delta+i\gamma,k,\lambda)&\leq C_{a,\delta'}\int^1_0s^{\delta'}(1-s)^{\delta-\delta'-1}A(\delta',k,\lambda s)ds\\
&\;\;\;\;\;\;\;\;\;+C_{a,\delta'}\int^1_0s^{\delta'}(1-s)^{\delta-\delta'-1}B(\delta',k,\lambda s)ds\\
&\leq C(a,\delta')\int^1_0s^{\delta'}(1-s)^{\delta-\delta'-1}(|\lambda|s k)^{-\delta'/2-1/4}ds\\
&\;\;\;\;\;\;\;\;\;+C_{a,\delta'}\int^1_0s^{\delta'}(1-s)^{\delta-\delta'-1}(|\lambda|s k)^{-\delta'/2+1/4}ds\\
&\leq C_{a,\delta'}(|\lambda|k)^{-\delta'/2-1/4}+C_{a,\delta'}(|\lambda|k)^{-\delta'/2+1/4}\\
&\leq C_{a,\delta'}(|\lambda|k)^{-\delta'/2+1/4}.
\end{align*}
Finally, we obtain that
\begin{align*}
&\|{\sup_{r\geq1}}^+|\bar{M}^{a}_rE_jf|^2\|_1\\
&\leq C\sum_{(\lambda,k)\in\Sigma_j}(L^{-1}A^2(\delta+i\gamma,k,\lambda)+LB^2(\delta+i\gamma,k,\lambda))\langle e_{(\lambda,k)}(f),f\rangle\\
&\leq C_{a,\delta'}\sup_{(\lambda,k)\in\Sigma_j}((|\lambda|k)^{-\delta'-1/2}L^{-1}+(|\lambda|k)^{-\delta'+1/2}L)\|f\|^2_2\\
&\leq C_{a,\delta'}2^{-j\delta'}\|f\|^2_2,
\end{align*}
by noting the fact that on $\Sigma_j$, $|\lambda k|\thicksim 2^j$ and choosing $L=2^{-j/2}$.

We finish the proof by establishing the two estimates (\ref{fundamental estimate 1}) and (\ref{fundamental estimate 2}). Defining
$$\mathcal{L}^{\delta'}_k(r)=\left(\frac{\Gamma(k+1)}{\Gamma(k+\delta'+1)}\right)^{1/2}e^{-\frac r2}r^{\frac {\delta'}2}L^{\delta'}_k(r).$$
Then we can write
$$\psi^{\delta'}_k(r)=\left(\frac{\Gamma(k+1)}{\Gamma(k+\delta'+1)}\right)^{1/2}\Gamma(\delta'+1)\Big(\frac{r^2}{2}\Big)^{-\frac{\delta'}{2}}\mathcal{L}^{\delta'}_k(\frac{r^2}{2}),$$
and noticing the fact $-\frac{d}{dr}L^{\delta'}_k(r)=-L^{\delta'+1}_{k-1}(r)$,
\begin{align*}
\frac{d}{dr}\psi^{\delta'}_k(r)&=\frac{\Gamma(k+1)\Gamma(\delta'+1)}{\Gamma(k+\delta'+1)}e^{-\frac14r^2}r\frac{d}{dr}L^{\delta'}_k(\frac{r^2}{2})\\
&\;\;\;\;\;+\frac{\Gamma(k+1)\Gamma(\delta'+1)}{\Gamma(k+\delta'+1)}(-\frac{r}{2})e^{-\frac14r^2}L^{\delta'}_k(\frac{r^2}{2})\\
&=-\frac{\Gamma(k+1)\Gamma(\delta'+1)}{\Gamma(k+\delta'+1)}\left(\frac{\Gamma(k+\delta')}{\Gamma(k)}\right)^{1/2}r\Big(\frac{r^2}{2}\Big)^{-\frac{\delta'+1}{2}}\mathcal{L}^{\delta'+1}_{k-1}(\frac{r^2}{2})\\
&\;\;\;\;\;+\left(\frac{\Gamma(k+1)}{\Gamma(k+\delta'+1)}\right)^{1/2}\Gamma(\delta'+1)(-\frac{r}{2})\Big(\frac{r^2}{2}\Big)^{-\frac{\delta'}{2}}\mathcal{L}^{\delta'}_k(\frac{r^2}{2})
\end{align*}
Noting that
$$\frac{\Gamma(k+1)}{\Gamma(k+\delta'+1)}\thicksim \Gamma(\delta'+1)k^{-\delta'}\;\mathrm{and}\;\frac{d}{dr}.$$ Hence
\begin{align*}
A^2(\delta',k,\eta)&=\eta^{-1/2}\int^\infty_{\sqrt{\eta}}|\psi^{\delta'}_k(r)|^2dr\\
&\leq C\Gamma^2(\delta'+1)\eta^{-1/2}k^{-\delta'}\int^\infty_{\sqrt{\eta}}\Big|\mathcal{L}^{\delta'}_k(\frac{r^2}{2})\Big|^2\Big(\frac{r^2}{2}\Big)^{{-\delta'}}dr\\
&\leq C_{\delta'}\eta^{-1/2}k^{-\delta'}\int^\infty_{{\eta}/2}|\mathcal{L}^{\delta'}_k(r)|^2r^{{-\delta'-1/2}}dr:=C_{\delta'}\cdot A.
\end{align*}
While
\begin{align*}
B^2(\delta',k,\eta)&=\eta^{1/2}\int^\infty_{\sqrt{\eta}}|\frac{d}{dr}\psi^{\delta'}_k(r)|^2dr\\
&\leq C_{\delta'}\eta^{1/2}k^{-\delta'}\int^\infty_{{\eta}/2}|\mathcal{L}^{\delta'+1}_{k-1}(r)|^2r^{{-\delta'-1/2}}dr\\
&\;\;\;\;\;+C_{\delta'}\eta^{1/2}k^{-\delta'}\int^\infty_{{\eta}/2}|\mathcal{L}^{\delta'}_k(r)|^2r^{{-\delta'+1/2}}dr:=C_{\delta'}\cdot B.
\end{align*}

To proceed with the proof, We need the following asymptotic property of Laguerre function which have been collected in Lemma 1.5.3 of \cite{Tha93}.
\begin{lemma}\label{lem:asymptotic property}
Let $k\geq1$. The Laguerre function satisfy
\begin{align*}
|\mathcal{L}^\delta_k(r)|\leq\left\{\begin{array}{cc}c(kr)^{\delta/2}& \;\mathrm{if}\; 0\leq t\leq 1/k\\
       c(kt)^{-1/4}& \;\mathrm{if}\; 1/k\leq t\leq k/2\\
       ck^{-1/4}(1-|k-t|)^{-1/4}& \;\mathrm{if}\; k/2\leq t\leq 3k/2\\
       ce^{-\gamma t}& \;\mathrm{if}\; t\geq 3k/2\end{array}\right.
\end{align*}

\end{lemma}

We now estimate $A$. Let us first deal with the case $\eta\leq k$ and $\eta/2\leq1/k$. Write
$$A=\eta^{-1/2}k^{-\delta'}\int^{1/k}_{{\eta}/2}+\int^{k/2}_{1/k}+\int^{3k/2}_{k/2}+\int^{\infty}_{3k/2}|\mathcal{L}^{\delta'}_k(r)|^2r^{{-\delta'-1/2}}dr.$$
Using the estimates in Lemma \ref{lem:asymptotic property} and the assumption  $\eta/2\leq1/k$, we obtain
\begin{align*}
\int^{1/k}_{{\eta}/2}|\mathcal{L}^{\delta'}_k(r)|^2r^{{-\delta'-1/2}}dr&\leq C\int^{1/k}_{{\eta}/2}(kr)^{\delta'}r^{{-\delta'-1/2}}dr\\
&\leq Ck^{\delta'}k^{-1/2}\leq C\eta^{-\delta'}k^{-1/2},
\end{align*}
and
\begin{align*}
\int^{k/2}_{1/k}|\mathcal{L}^{\delta'}_k(r)|^2r^{{-\delta'-1/2}}dr&\leq C\int^{k/2}_{1/k}(kr)^{-1/2}r^{{-\delta'-1/2}}dr\\
&\leq Ck^{\delta'}k^{-1/2}\leq C\eta^{-\delta'}k^{-1/2}.
\end{align*}
For the third and forth integrals, we use the assumption $\eta\leq k$ and the estimates in Lemma \ref{lem:asymptotic property}
\begin{align*}
\int^{3k/2}_{k/2}|\mathcal{L}^{\delta'}_k(r)|^2r^{{-\delta'-1/2}}dr&\leq C\int^{3k/2}_{k/2}k^{-1/2}|k-r|^{-1/2}r^{{-\delta'-1/2}}dr\\
&\leq Ck^{-1/2}k^{-\delta'-1/2}k^{1/2}\leq C\eta^{-\delta'}k^{-1/2}.
\end{align*}
and
\begin{align*}
\int^{\infty}_{3k/2}|\mathcal{L}^{\delta'}_k(r)|^2r^{{-\delta'-1/2}}dr&\leq C\int^{\infty}_{3k/2}e^{-2\gamma r}r^{{-\delta'-1/2}}dr\\
&\leq Ck^{-\delta'-1/2}\leq C\eta^{-\delta'}k^{-1/2}.
\end{align*}
The case $\eta\leq k$ and $\eta/2\geq1/k$ follows easily, since
\begin{align*}
\int^{k/2}_{\eta/2}|\mathcal{L}^{\delta'}_k(r)|^2r^{{-\delta'-1/2}}dr&\leq C\int^{k/2}_{\eta/2}(kr)^{-1/2}r^{{-\delta'-1/2}}dr\\
&\leq C\eta^{-\delta'}k^{-1/2}.
\end{align*}
The other two cases $k\leq \eta\leq 3k$ and $\eta\geq3k$ can be estimated similarly, and bounded by $C\eta^{-\delta'}k^{-1/2}.$ We leave them for the readers.
Finally, in all cases, we obtain
$$A\leq C(\eta k)^{-\delta'-1/2}$$
and hence prove the estimate (\ref{fundamental estimate 1}).

The estimate of $B$ can be done similarly.  For the first part in $B$, when $k=1$,  it follows easily since $\mathcal L^{\delta'+1}_0(r)=(\Gamma(\delta'+2))^{-1/2}e^{-r/2}r^{\frac{\delta'+1}{2}}$. Hence  it suffices to prove
$$\eta^{1/2}k^{-\delta'}\int^\infty_{{\eta}/2}|\mathcal{L}^{\delta'+1}_{k}(r)|^2r^{{-\delta'-1/2}}dr\leq C(\eta k)^{-\delta'+1/2}$$
for $k\geq1$.
We only need to deal with the following estimate in the case $\eta\leq k$ and $\eta/2\leq1/k$,
\begin{align*}
\int^{1/k}_{{\eta}/2}|\mathcal{L}^{\delta'+1}_{k}(r)|^2r^{{-\delta'-1/2}}dr&\leq C\int^{1/k}_{{\eta}/2}(kr)^{\delta'+1}r^{{-\delta'-1/2}}dr\\
&\leq Ck^{\delta'+1}k^{-3/2}\leq C\eta^{-\delta'}k^{-1/2},
\end{align*}
since other cases follows by the same arguments used in the estimate of $A$ and bounded by $C\eta^{-\delta'}k^{-1/2}$. Therefore
$$\eta^{1/2}k^{-\delta'}\int^\infty_{{\eta}/2}|\mathcal{L}^{\delta'+1}_{k}(r)|^2r^{{-\delta'-1/2}}dr\leq C\eta^{-\delta'}k^{-1/2}\eta^{1/2}k^{-\delta'}\leq C(\eta k)^{-\delta'+1/2}.$$
The estimate of the second integral in $B$ is dealt with similarly. But the calculation is different, so we give a proof.
Write it as
$$\eta^{1/2}k^{-\delta'}\int^{1/k}_{{\eta}/2}+\int^{k/2}_{1/k}+\int^{3k/2}_{k/2}+\int^{\infty}_{3k/2}|\mathcal{L}^{\delta'+1}_{k}(r)|^2r^{{-\delta'+1/2}}dr.$$
Using the estimates in Lemma \ref{lem:asymptotic property} and the assumption  $\eta/2\leq1/k$, we obtain
\begin{align*}
\int^{1/k}_{{\eta}/2}|\mathcal{L}^{\delta'}_{k}(r)|^2r^{{-\delta'+1/2}}dr&\leq C\int^{1/k}_{{\eta}/2}(kr)^{\delta'}r^{{-\delta'+1/2}}dr\\
&\leq Ck^{\delta'}k^{-3/2}\leq C\eta^{-\delta'}k^{1/2},
\end{align*}
and together with $\eta\leq k$,
\begin{align*}
\int^{k/2}_{1/k}|\mathcal{L}^{\delta'}_k(r)|^2r^{{-\delta'+1/2}}dr&\leq C\int^{k/2}_{1/k}(kr)^{-1/2}r^{{-\delta'+1/2}}dr
\end{align*}
which is smaller than
$$Ck^{-1/2}k^{1-\delta'}\leq k^{1/2}\eta^{-\delta'}\;\mathrm{if}\;\delta'\leq 1$$
and
$$Ck^{-1/2}k^{-1+\delta'}\leq k^{1/2}\eta^{-\delta'}\;\mathrm{if}\;\delta'\geq 1$$
For the third and forth integrals, we use the assumption $\eta\leq k$ and the estimates in Lemma \ref{lem:asymptotic property}
\begin{align*}
\int^{3k/2}_{k/2}|\mathcal{L}^{\delta'}_k(r)|^2r^{{-\delta'+1/2}}dr&\leq C\int^{3k/2}_{k/2}k^{-1/2}|k-r|^{-1/2}r^{{-\delta'+1/2}}dr\\
&\leq Ck^{-1/2}k^{-\delta'+1/2}k^{1/2}\leq C\eta^{-\delta'}k^{1/2}.
\end{align*}
and
\begin{align*}
\int^{\infty}_{3k/2}|\mathcal{L}^{\delta'}_k(r)|^2r^{{-\delta'-1/2}}dr&\leq C\int^{\infty}_{3k/2}e^{-2\gamma r}r^{{-\delta'+1/2}}dr\\
&\leq Ck^{-\delta'+1/2}\leq C\eta^{-\delta'}k^{1/2}.
\end{align*}
Other cases is estimated similarly. Hence we obtain that the second integral in $B$ is also bounded by $C(\eta k)^{-\delta'+1/2}$, thus conclude that
$$B\leq C(\eta k)^{-\delta'+1/2}.$$
\end{proof}

\section{Individual ergodic theorems}

In this section we apply the maximal inequality to study pointwise convergence. We first need an appropriate analogue for the noncommutative setting of the usual almost everywhere convergence. This is the almost uniform convergence introduced by Lance \cite{Lan76} (see also \cite{JuXu06}).

Let $(x_{\lambda})_{\lambda \in \Lambda}$ be a family of elements in
$L_p(\mathcal{M}).$ Recall that $(x_{\lambda})_{\lambda\in\Lambda}$ is said to converge almost uniformly to $x,$ abbreviated by $x_{\lambda} \xrightarrow{a.u} x,$ if for every $\epsilon>0$ there exists a projection $e \in \mathcal{M}$ such that
$$ \tau(1 - e) < \epsilon \quad \text{and} \quad \lim_{\lambda}\|e ( x_{\lambda} - x )\|_{\infty} = 0.$$
Also, $( x_{\lambda})_{\lambda \in \Lambda}$ is said to converge bilaterally almost uniformly to $x
,$ abbreviated by $x_{\lambda} \xrightarrow{b.a.u} x,$ if for every $\epsilon>0$ there
is a projection $e \in \M$ such that
$$ \tau( 1 - e ) < \epsilon \quad \text{and} \quad \lim_{\lambda} \|e ( x_{\lambda} - x )e \|_{\infty} = 0.$$

Obviously, if $x_{\lambda}\xrightarrow{a.u} x$ then $x_{\lambda}\xrightarrow{b.a.u} x.$ On the other hand, in the commutative case, these two convergences are equivalent to the usual almost everywhere convergence in terms of Egorov's theorem. However they are different in the noncommutative setting.

As in \cite{JuXu06}, in order to deduce the pointwise convergence theorems from the
corresponding maximal inequalities, it is convenient to use the closed subspace $L_p(\mathcal{M}; c_0)$ of
$L_p(\mathcal{M};\ell_{\infty}).$ Recall that $L_p(\mathcal{M}; c_0)$ is defined
as the space of all sequences $(x_n) \in L_p(\mathcal{M})$ such that
there are $a, b\in L_{2p}(\mathcal{M})$ and $(y_n)\subset \mathcal{M}$
verifying
$$ x_n = a y_n b\quad \text{and}\quad \lim_ n \|y_n \|_{\infty} = 0.$$
Similarly, for the study of the {\it a.u} convergence, we use the closed subspace $L_p(\mathcal{M}; c^c_0)$ of
$L_p(\mathcal{M};\ell^c_{\infty})$, which is defined to be the space of all sequences $(x_n) \in L_p(\mathcal{M})$ such that
there are $b\in L_{p}(\mathcal{M})$ and $(y_n)\subset \mathcal{M}$
verifying
$$ x_n = y_n b\quad \text{and}\quad \lim_ n \|y_n \|_{\infty} = 0.$$

The following lemma will be useful for our study of pointwise convergence
theorems in noncommutative setting (see \cite{DeJu04}).

\begin{lemma}\label{lem:Deju}
\begin{enumerate}[{\rm(i)}]
\item If $1 \leq p <\infty$ and $\{ x_n \} \in L_p(\mathcal{M}; c_0),$ then $x_n \xrightarrow{b.a.u} 0.$
\item If $2 < p < \infty$ and $\{ x_n \} \in L_p (\M; c^c_0),$ then $x_n \xrightarrow{a.u} 0.$
\end{enumerate}
\end{lemma}

The second lemma we need is about the uniform pointwise estimates for the spherical functions, which will be used to deduce individual ergodic theorem on dense subspaces.
\begin{lemma}\label{lem:pointwise spectral estimates}
{\rm (i)}. Let $n\geq1$, fix $0<\varepsilon\leq1$, $N\in\mathbb{N}$ and $r>1$. Let $\Sigma_{\varepsilon,N}=\{(\lambda,k):\;\varepsilon\leq|\lambda|\leq N,\;k\leq N\}$. Then
\begin{align}\label{pointwise spectral estimates1}
\sup_{\zeta\in\Sigma_{\varepsilon,N}}\left|\varphi_\zeta(r)\right|\leq Ce^{-\gamma\varepsilon r^2},
\end{align}
where $C$ and $\gamma$ are positive constants independent of $r$.

{\rm (ii)}. Fix $\varepsilon>0$, and let $\Sigma'_\varepsilon=\{u:\;\varepsilon\leq u\leq \varepsilon^{-1}\}$. Then
\begin{align}\label{pointwise spectral estimates2}
\sup_{\zeta\in\Sigma'_{\varepsilon}}\left|\varphi_\zeta(r)\right|\leq \frac{C}{\varepsilon^{3}}r^{-n+1/2},
\end{align}
where $C$ is a positive constant independent of $r$.
\end{lemma}

The estimates in (i) follows from the asymptotic properties of Laguerre polynomials of type $n-1$, see e.g. formula 5.1.14 of \cite{Sze67}. The estimates in (ii) follows from the standard expansion of the Bessel functions at infinity. See for instance Section 3 of \cite{NeTh97} for detailed information.

Now let us present the main result in this section.

\begin{theorem}\label{thm:indvidual ergodic theorem}
Let $n>1$.
Let  $x\in L_p(\M)$.  We have
\begin{enumerate}[\rm (i)]
\item The family $\alpha(\sigma_r)x$ converges to $F(x)$ b.a.u. for $(2n-1)/(2n-2)<p\leq2$ and a.u. for $2<p<\infty$ as $r\rightarrow\infty$.

\item The family $\alpha(\sigma_r)x$ converges to $x$ b.a.u. for $(2n-1)/(2n-2)<p\leq2$ and a.u. for $2<p<\infty$ as $r\rightarrow0$.

\item The family $\alpha(\bar{\sigma}_r)x$ converges to $F(x)$ b.a.u. for $2n/(2n-1)<p\leq2$ and a.u. for $2<p<\infty$ as $r\rightarrow\infty$.

\end{enumerate}

\end{theorem}

\begin{proof}
(i). Let us first prove the case $p=2$. Note that $L_2(\M)$ is  the orthogonal sum of three closed subspace: $\mathcal{H}_1=$ the space of operators invariant under each $\alpha(\sigma_r)$, $\mathcal{H}_{\Sigma}=$ the space of operators whose spectral measure is supported in the union of the Laguerre and Bessel spectrum, and finally, $\mathcal{H}_0=$ the space of operators in the kernel of each $\alpha(\sigma_r)$. Clearly, in the first and the third subspace we have $\alpha(\sigma_r)x-F(x)=0$. Hence by Lemma \ref{lem:Deju}, it suffices to prove $(\alpha(\sigma_r)x)_{r>1}\in L_2(\M;c_0)$ for $x\in \mathcal{H}_{\Sigma}$. By the maximal ergodic theorem---Theorem \ref{thm:maximal radial average mpc}, it suffices to prove $(\alpha(\sigma_r)x)_{r>1}\in L_2(\M;c_0)$ for $x$ in some dense set of $\mathcal{H}_{\Sigma}$. {Indeed, suppose we have $(\alpha(\sigma_r)y)_{r>1}\in L_2(\M;c_0)$ for all $y$ in a dense set of $\mathcal{H}_{\Sigma}$}. Then for fixed $x\in \mathcal{H}_{\Sigma}$, for any $\delta>0$, there exists a $y$ in the dense subset of $\mathcal{H}_{\Sigma}$, such that $\|x-y\|_2\leq\delta$. Therefore by maximal inequality (\ref{maximal radial average mpc})
\begin{align*}
\|(\alpha(\sigma_r)x)_{r>1}-(\alpha(\sigma_r)y)_{r>1}\|_{L_2(\ell_{\infty})}\leq C\|x-y\|_2\leq C\delta.
\end{align*}
Since $\delta$ is arbitrary, $(\alpha(\sigma_r)x)_{r>1}$ is in the closure of $L_2(\M;c_0)$, thus belongs to $L_2(\M;c_0)$ because $L_2(\M;c_0)$ is closed.

For $\varepsilon>0$ and $N>0$, let $\mathcal{H}^N_{\varepsilon}$ be the subspace of operators $x\in L_2(\M)$ whose spectral measure $\langle de(x),x\rangle$ is supported in $\Sigma_{\varepsilon,N}$ defined in Lemma \ref{lem:pointwise spectral estimates} and $\mathcal{H}'_{\varepsilon}$ be the subspace of operators whose spectral measure is supported in $\Sigma'_{\varepsilon}$ defined in Lemma \ref{lem:pointwise spectral estimates}. The dense set of $\mathcal{H}_{\Sigma}$ we shall consider is
$\bigcup_{\varepsilon,N}\mathcal{H}^N_{\varepsilon}+\mathcal{H}'_{\varepsilon}$. For $x\in \mathcal{H}^N_{\varepsilon}+\mathcal{H}'_{\varepsilon}$, by spectral decomposition (\ref{spectral decomposition}) as well as the spectral estimates (\ref{pointwise spectral estimates1}) and (\ref{pointwise spectral estimates2}), we have
$$\|\alpha(\sigma_r)x\|_2\leq \sup_{\zeta\in\Sigma_{\varepsilon,N}\cup\Sigma'_{\varepsilon}}|\varphi_{\zeta}(r)|\|x\|_2\leq C\max(e^{-\gamma\varepsilon r^2}, \varepsilon^{-3}r^{-n+1/2})\|x\|_2.$$
Then the fact that  the space $L_2(\M;c_0)$ is complete yields that $(\alpha(\sigma_r)x)_{r>1}\in L_2(\M;c_0)$, since
$$(\alpha(\sigma_r)x)_{r>1}=\int^\infty_1(\alpha(\sigma_r)x)_{r=t}dt$$
and
$$\int^\infty_1\max(e^{-\gamma\varepsilon t^2}, \varepsilon^{-3}t^{-n+1/2})dt\leq C<\infty.$$

For other cases $p\neq2$, using again the maximal ergodic theorem---Theorem \ref{thm:maximal radial average mpc}, it suffices to prove $(\alpha(\sigma_r)x)_{r>1}\in L_p(\M;c_0)$ for $x\in L_1(\mathcal M)\cap \mathcal M$, which is a dense subset of $L_p(\mathcal M)$. Without loss of generality, we assume $p<2$. Find $q$ and $\theta$ such that $(2n-1)/(2n-2)<q<p$ and $1/p=(1-\theta)/q+\theta/2$. Let $x\in L_1(\mathcal M)\cap \mathcal M$. For any $1<t<s$, by maximal inequality (\ref{maximal radial average mpc}),
\begin{align*}
\|(\alpha(\sigma_r)x)_{t<r<s}\|_{L_p(\ell_{\infty})}&\leq \|(\alpha(\sigma_r)x)_{t<r<s}\|^{1-\theta}_{L_q(\ell_{\infty})}\|(\alpha(\sigma_r)x)_{t<r<s}\|^\theta_{L_2(\ell_{\infty})}\\
&\leq C_q^{1-\theta}\|x\|_q\|(\alpha(\sigma_r)x)_{t<r<s}\|^\theta_{L_2(\ell_{\infty})},
\end{align*}
which tends to 0 when $t$ tends to $\infty$ using the result in the case $p=2$. Hence $\alpha(\sigma_r)x)_{r>1}$ is approximated by $\alpha(\sigma_r)x)_{1<r<t}$'s, therefore belongs to $L_p(\M;c_0)$.

(ii). By Lemma \ref{lem:Deju}, it suffices to prove $(\alpha(\sigma_r)x-x)_{0<r\leq1}$ belongs to $L_p(\M;c_0(0,1])$ (with respect to $t\rightarrow0$). As before, by maximal inequality (\ref{maximal radial average mpc}), it suffices show $(\alpha(\sigma_r)x-x)_{0<r\leq1}$ belongs to $L_p(\M;c_0(0,1])$ for $x$ in a dense subset of $L_p(\mathcal M)$. This time, we use the dense subset whose element is of the form $x=\alpha(\phi)y$ with $\phi\in C_c^\infty(H^n)$ and $y\in L_1(\M)\cap \M$ since $\mathcal{D}$ and $L_1(\M)\cap \M$ are dense in $L_p(\M)$. Using the arguments in Theorem \ref{thm:mean ergodic theorem},
$$\|\alpha(\sigma_r)\alpha(\phi)y-\alpha(\phi)y\|_{\infty}\leq\|\sigma_r\ast\phi-\phi\|_1\|y\|_{\infty}\leq C_{\phi}r\|y\|_\infty.$$
Now we finish the proof by interpolation. Let $0<t<s\leq1$. Take $(2n-1)/(2n-2)<q<p$, by the maximal inequality (\ref{maximal radial average mpc}),
\begin{align*}
&\|(\alpha(\sigma_r)x-x)_{t<r<s}\|_{L_p(\ell_{\infty})}\\
&\leq \|(\alpha(\sigma_r)x-x)_{t<r<s}\|^{q/p}_{L_q(\ell_{\infty})}\sup_{t<r<s}\|\alpha(\sigma_r)x-x\|^{1-q/p}_{\infty}\\
&\leq C_{p,q}\|x\|^{q/p}_q(C_\phi s\|y\|_{\infty})^{1-q/p},
\end{align*}
which which tends to 0 when $s$ tends to 0. Whence $\alpha(\sigma_r)x)_{0<r\leq1}$ is approximated by $\alpha(\sigma_r)x)_{s<r\leq1}$'s, therefore belongs to $L_p(\M;c_0(0,1])$.

(iii). In this case, we use Theorem \ref{thm:maximal radial average reduced} instead of Theorem \ref{thm:maximal radial average}, and the same arguments used in (i) are available. 

\end{proof}

\noindent \textbf{Acknowledgement.} The author would like to thank Sundaram Thangavelu for helpful discussion in the proof of Proposition \ref{pro:maximal psi a2 2}. The work is partially supported by the NSF of China-11601396, Funds for Talents of China-413100002 and 1000 Young Talent Researcher Programm of China-429900018-101150(2016).

\end{document}